\documentclass[letterpaper,11pt,twoside,keywordsasfootnote,addressatend,noinfoline]{article}
\usepackage{fullpage}
\usepackage[english]{babel}
\usepackage{amssymb}
\usepackage{amsmath}
\usepackage{theorem}
\usepackage{epsfig}
\usepackage{subfigure}
\usepackage{mathrsfs}
\usepackage{color}
\usepackage{imsart}

\newtheorem{theorem}{Theorem}
\newtheorem{lemma}[theorem]{Lemma}

\newenvironment{proof}{\noindent{\bf Proof.}}{\hspace*{2mm}~$\square$}

\newcommand{\N}{\mathbb{N}}
\newcommand{\Z}{\mathbb{Z}}
\newcommand{\R}{\mathbb{R}}
\newcommand{\Lat}{\mathscr{L}}
\newcommand{\C}{\mathscr{C}}

\newcommand{\norm}[1]{|\!|#1|\!|}
\newcommand{\ind}{\mathbf{1}}
\newcommand{\ep}{\epsilon}
\newcommand{\n}{\hspace*{-5pt}}

\DeclareMathOperator{\card}{card}
\DeclareMathOperator{\vol}{Vol}
\DeclareMathOperator{\de}{Det}
\DeclareMathOperator{\poisson}{Poisson}
\DeclareMathOperator{\exponential}{Exponential}

\begin{document}

\begin{frontmatter}
\title     {Durrett-Levin spatial model of allelopathy}
\runtitle  {Durrett-Levin spatial model of allelopathy}
\author    {Nicolas Lanchier\thanks{Nicolas Lanchier was partially supported by NSF grant CNS-2000792.} and Max Mercer}
\runauthor {Nicolas Lanchier and Max Mercer}
\address   {School of Mathematical and Statistical Sciences \\ Arizona State University \\ Tempe, AZ 85287, USA. \\ nicolas.lanchier@asu.edu \\ mamerce1@asu.edu}

\maketitle

\begin{abstract} \ \
 Allelopathy refers to a type~$0/-$ biological interaction that is neutral for a so-called inhibitory species but detrimental for a so-called susceptible species.
 To model this type of interaction in a spatially-structured environment, Durrett and Levin introduced a variant of the multitype contact process in which the death rate of the susceptible species is density-dependent, increasing with the local density of the inhibitory species.
 Their work combines mean-field analysis and simulations of the spatial model, and our main objective is to give rigorous proofs of some of their conjectures.
 In particular, we give a complete description of the behavior of the mean-field model, including the global stability of the fixed points.
 Our main results for the interacting particle system show the existence of two regimes depending on the relative fitness of the individuals.
 When the inhibitory species is the superior competitor, the inhibitory species always wins, whereas when the susceptible species is the superior competitor, the susceptible species wins if and only if the inhibitory effects do not exceed some critical threshold.
 We also prove that, at least in dimensions~$d \geq 3$, the transition between these two regimes is continuous in the sense that, when both species are equally fit, the inhibitory species wins even in the presence of extremely weak inhibitory effects.
\end{abstract}



\end{frontmatter}

%%%%%%%%%%%%%%%%%%%%%%%%%%%%%%%%%%%%%%%%%%%%%%%%%%%%%%%%%%%%%%%%%%%%%%%%%%%%%%%%%%%%%%%%%%%%%%%%%%%%%%%%%%%%%%%%%%%%%%%%%%%%%%%%%%%%%%%%%%%%%%%%%%%%%%%%%%%%%%%%%%%%%%%%%%%%%%%%%%%%%%%%%%%%%%%%%%

\section{Introduction}
\label{sec:intro}
 Biological interactions among species can be beneficial~$(+)$, neutral~$(0)$ or detrimental~$(-)$ for each of the interacting species.
 Biological interactions that are neutral for one species and detrimental for another species are referred to as amensalism.
 Allelopathy is a particular case of amensalism called antibiosis amensalism in which one species produces toxic biochemicals called allelochemicals that inhibit the growth and survival of the other species.
 The species producing the allelochemicals is called the inhibitory species while the other species is called the susceptible species.
 Because the inhibitory species is immune to the allelochemicals it produces, this mechanism provides a selective advantage to the inhibitory species.
 To model allelopathy in a spatially-structured environment, Durrett and Levin~\cite{Durrett_Levin_1997} introduced an interacting particle system in which each lattice point can be empty or occupied by an individual of one of the two species.
 The system is characterized by the natural birth and death rates of the species as well as the strength of the inhibitory effects.
 The spatial component is included in the form of local interactions assuming that the offspring are sent from the parent's location to nearby lattice points but also that the individuals of the inhibitory species only affect nearby susceptible individuals.
 Their model assumes in addition that offspring can only survive when sent to an empty site, which also models competition for space. \\
\indent
 More precisely, the Durrett-Levin model is a variant of the multitype contact process~\cite{Neuhauser_1992} that models allelopathy by including a density-dependent death rate for the susceptible species.
 Like in the multitype contact process, each site of the~$d$-dimensional lattice is in state
 $$ \begin{array}{rcrcrcl}
     0 & \n = \n & \hbox{empty} && 1 & \n = \n & \hbox{occupied by a type~1 individual~(inhibitory species)} \vspace*{4pt} \\
       & \n   \n & \hbox{or}    && 2 & \n = \n & \hbox{occupied by a type~2 individual~(susceptible species)} \end{array} $$
 so the state of the entire system at time~$t$ is a spatial configuration
 $$ \xi_t : \Z^d \longrightarrow \{0, 1, 2 \} \quad \hbox{where} \quad \xi_t (x) = \hbox{state at site~$x$ at time~$t$}. $$
 To describe the dynamics~(local interactions), we let
 $$ N_x = \{y \in \Z^d : 0 < \norm{x - y} \leq M \} \quad \hbox{for all} \quad x \in \Z^d $$
 be the neighborhood of site~$x$, where~$\norm{\cdot}$ refers to the Euclidean norm and~$M$ is interpreted as a dispersal range.
 For every species~$i$, site~$x$ and configuration~$\xi$, we let
 $$ f_i (x, \xi) = \frac{1}{\card (N_x)} \ \sum_{y \in N_x} \ind \{\xi (y) = i \} $$
 be the fraction of neighbors of site~$x$ that are occupied by an individual of type~$i$.
 Then, the spatial allelopathic model evolves according to the local transitions
\begin{equation}
\label{eq:IPS}
\begin{array}{rclcrcl}
  0 \to 1 & \hbox{at rate} & \beta_1 f_1 (x, \xi) & \qquad & 1 \to 0 & \hbox{at rate} & 1  \vspace*{4pt} \\
  0 \to 2 & \hbox{at rate} & \beta_2 f_2 (x, \xi) & \qquad & 2 \to 0 & \hbox{at rate} & 1 + \gamma f_1 (x, \xi). \end{array}
\end{equation}
 The two transitions on the left indicate that type~$i$ individuals give birth at rate~$\beta_i$ to individuals of their own type.
 The offspring is sent to a site chosen uniformly at random from the parent's neighborhood, and takes place in the system if and only if the target site is empty, which models competition for space.
 The two transitions on the right indicate that the individuals die naturally at rate one regardless of their type but the overall death rate of the susceptible individuals is density-dependent, increasing with respect to the local density of the inhibitory species, which models allelopathy whose strength is measured by the parameter~$\gamma$.
 The particular case~$\gamma = 0$ corresponds to the multitype contact process which was introduced by Neuhauser~\cite{Neuhauser_1992}.
 Another spatial model of allelopathy based on the framework of interacting particle systems was studied more recently by Lanchier~\cite{Lanchier_2005} using a different modeling approach.
 His model is also a variant of Neuhauser's multitype contact process but, contrary to the Durrett-Levin model, the death rates are both spontaneous.
 Instead, the model includes an additional state interpreted as a frozen site, and allelopathy is modeled by assuming that individuals of the inhibitory species produce metabolites at their location, leaving after their death a frozen site that remains inaccessible to the susceptible species~(but not to the inhibitory species) for a random amount of time. \\
\indent
 Returning to the Durrett-Levin model~\eqref{eq:IPS}, assuming that the population is homogeneously mixing, and letting~$u_1$ be the density of the inhibitory species and~$u_2$ be the density of the susceptible species, the process is described by the following deterministic mean-field model:
\begin{equation}
\label{eq:mean-field}
\begin{array}{rcl}
  u_1' & \n = \n & F_1 (u_1, u_2) = \beta_1 u_1 (1 - u_1 - u_2) - u_1 \vspace*{4pt} \\
  u_2' & \n = \n & F_2 (u_1, u_2) =  \beta_2 u_2 (1 - u_1 - u_2) - (1 + \gamma u_1) u_2.
\end{array}
\end{equation}
 Durrett and Levin~\cite{Durrett_Levin_1997} proved that, whenever
\begin{equation}
\label{eq:bistable}
\beta_1, \beta_2 > 1 \quad \hbox{and} \quad \beta_1 < \beta_2 < (1 + \gamma) \beta_1 - \gamma,
\end{equation}
 the system has two locally stable boundary fixed points, one with a positive density of~1s and one with a positive density of~2s, as well as one interior fixed point.
 Their numerical simulations also show that the system is bistable in the sense that the interior fixed point is a saddle point and the system converges to one of the two boundary fixed points, which is also suggested by the solution curves in Figure~\ref{fig:MF}.
\begin{figure}[t!]
\centering
\scalebox{0.35}{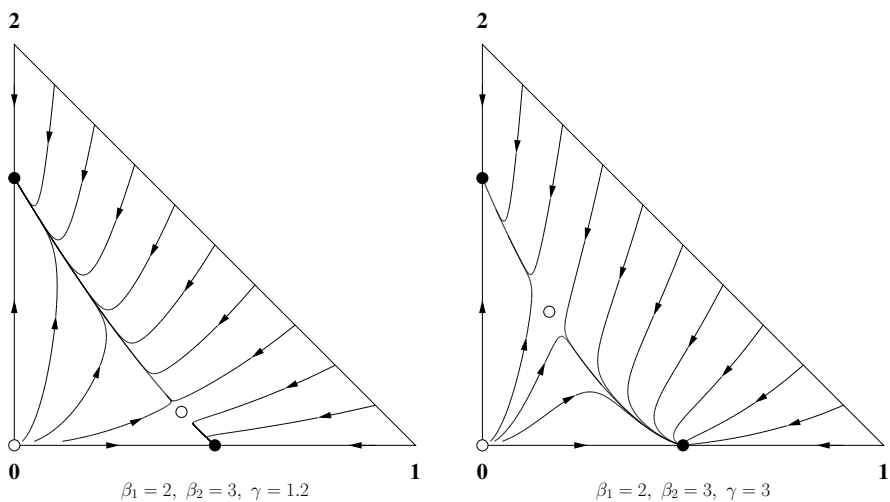}
\caption{\upshape{
 Solution curves of the mean-field model for two different values of~$\gamma$ and the same birth rates.
 The boundary fixed points are the same but the basin of attraction of the boundary fixed point in which only the inhibitory/susceptible species is present increases/decreases with~$\gamma$.}}
\label{fig:MF}
\end{figure}
 In particular, in the parameter region~\eqref{eq:bistable}, none of the species can invade the other species in its equilibrium, and the outcome of the competition depends on the initial density of each species.
 In Section~\ref{sec:mean-field}, we will identify all the fixed points, study their local stability, and use the Bendixson--Dulac theorem to give a rigorous proof of this result and a complete picture of the phase structure of the mean-field model. \\
\indent
 Our analysis together with numerical simulations shows that the spatial model has a number of similarities with its mean-field approximation but also one major difference:
 the parameter region where the mean-field model is bistable is turned into a region where there is a strong type that wins even when starting at low density in the spatial model.
 In particular, the outcome of the competition depends on the parameters of the system but not on the initial densities.
 To state our results and explore the phase structure, we say that species~$i$
 $$ \begin{array}{rcrcl}
    \hbox{survives} & \ \hbox{when} & \liminf_{t \to \infty} P (\xi_t (x) = i) > 0 & \hbox{for all} & x \in \Z^d \vspace*{4pt} \\
    \hbox{dies out} & \ \hbox{when} &    \lim_{t \to \infty} P (\xi_t (x) = i) = 0 & \hbox{for all} & x \in \Z^d \end{array} $$
 and wins when it survives while the other species dies out.
 In addition, we say that both species coexist when they both survive.
 Letting~$\beta_c$ be the critical value of the contact process described below, the behavior of the allelopathic model outside the parameter region~$\beta_2 \geq \beta_1 > \beta_c$ can be easily deduced from coupling arguments to compare the spatial allelopathic model with Harris' contact process and Neuhauser's multitype contact process.
 In the absence of one species, the other species, say species~1, behaves like Harris' contact process~\cite{Harris_1974}, the interacting particle system in which sites can be either~0~=~empty or~1~=~occupied, with local transitions
 $$ \begin{array}{rclcrcl}
      0 \to 1 & \hbox{at rate} & \beta_1 f_1 (x, \xi) & \quad \hbox{and} \quad & 1 \to 0 & \hbox{at rate} & 1. \end{array} $$
 As one of the most popular models of interacting particle systems, this process was intensively studied, and we refer to Liggett~\cite{Liggett_1985,Liggett_1999} for a review of the main results.
 In particular, there exists a nondegenerate~(positive and finite) critical value~$\beta_c$ that depends on both the spatial dimension and the range of the interactions such that, starting with infinitely many~1s,
 $$ \begin{array}{rcl}
    \beta_1 \leq \beta_c & \Longrightarrow & \hbox{species 1 dies out} \vspace*{4pt} \\
       \beta_1 > \beta_c & \Longrightarrow & \hbox{species 1 survives}. \end{array} $$
 The proof of extinction at the critical value~$\beta_c$ is due to Bezuidenhout and Grimmett~\cite{Bezuidenhout_Grimmett_1990}.
 To study the competition between both species, we now assume that the process starts from a translation invariant measure with a positive density of both species, so the initial configuration contains infinitely many individuals of both types.
 In this case, the set of occupied sites is dominated by a contact process with parameter~$\beta_1 \vee \beta_2$ from which it follows that
\begin{equation}
\label{eq:extinction}
\begin{array}{rcl} \beta_1 \leq \beta_c \quad \hbox{and} \quad \beta_2 \leq \beta_c  & \Longrightarrow & \hbox{both species die out}. \end{array}
\end{equation}
 Using another result of Bezuidenhout and Grimmett~\cite{Bezuidenhout_Grimmett_1991} about the exponential decay of the subcritical contact process implies that, when a species is subcritical, say species~2, it dies out fast enough that species~1 behaves eventually as in the absence of species~2.
 In particular,
\begin{equation}
\label{eq:contact}
\begin{array}{rcl}
\hbox{(a)} \ \ \beta_1 > \beta_c \quad \hbox{and} \quad \beta_2 \leq \beta_c & \Longrightarrow & \hbox{species 1 wins} \vspace*{4pt} \\
\hbox{(b)} \ \ \beta_1 \leq \beta_c \quad \hbox{and} \quad \beta_2 > \beta_c & \Longrightarrow & \hbox{species 2 wins}. \end{array}
\end{equation}
 We now assume that the two birth rates are supercritical, in which case the long-term behavior depends on the relative fitness of the two species and the strength of allelopathy.
 Standard coupling arguments show that the set of~1s is nondecreasing with respect to the birth rate~$\beta_1$ and the parameter~$\gamma$ and nonincreasing with respect to the birth rate~$\beta_2$, and similarly but with the monotonicity reversed for the set of~2s.
 This implies that, two of the parameters being fixed, there is at most one phase transition from extinction to survival of a given species in the direction of the third parameter.
 In addition, in the absence of inhibitory effects~$\gamma = 0$, the process reduces to the multitype contact process, and a result of Neuhauser~\cite{Neuhauser_1992} based on duality techniques shows that the species with the larger birth rate wins.
\begin{figure}[t!]
\centering
\scalebox{0.35}{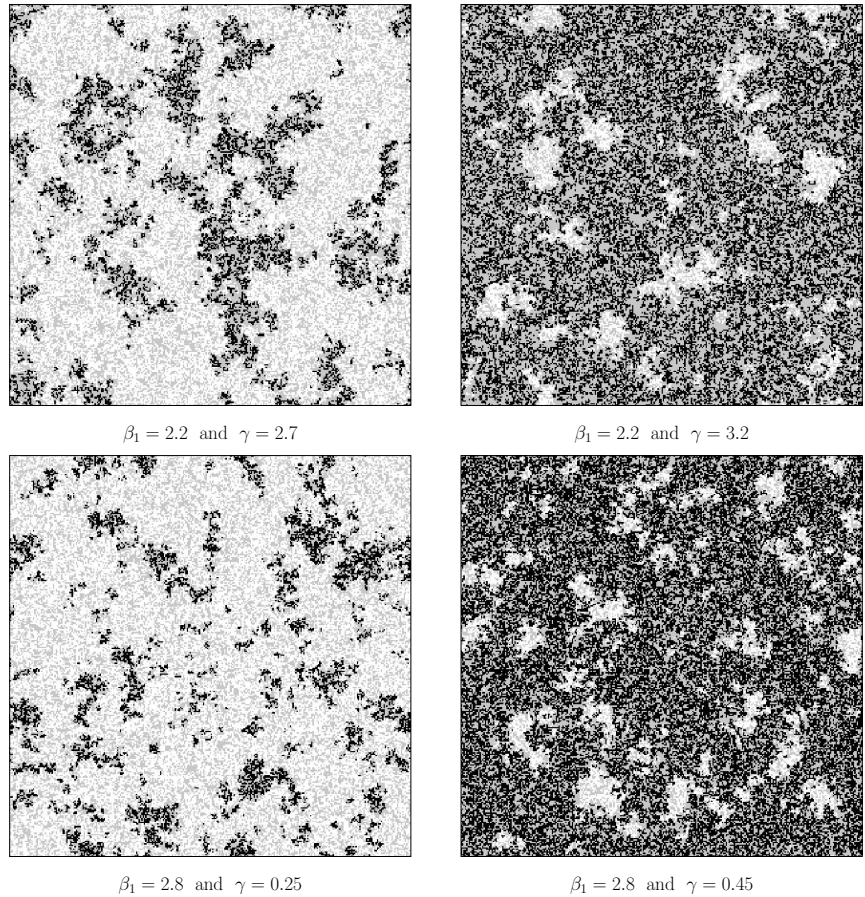}
\caption{\upshape{
 Snapshots of the process on the~$300 \times 300$ torus at time~200.
 In all the pictures, the birth rate of the susceptible species is equal to~$\beta_2 = 3$, the black sites represent the inhibitory species and the white sites the susceptible species.
 Note that, while decreasing~$\beta_1$ and increasing~$\gamma$, the clusters of the inhibitory species are less dense but the susceptible species cannot easily invade these clusters due to the strength of allelopathy measured~$\gamma$.}}
\label{fig:IPS}
\end{figure}
\begin{figure}[t!]
\centering
\scalebox{0.35}{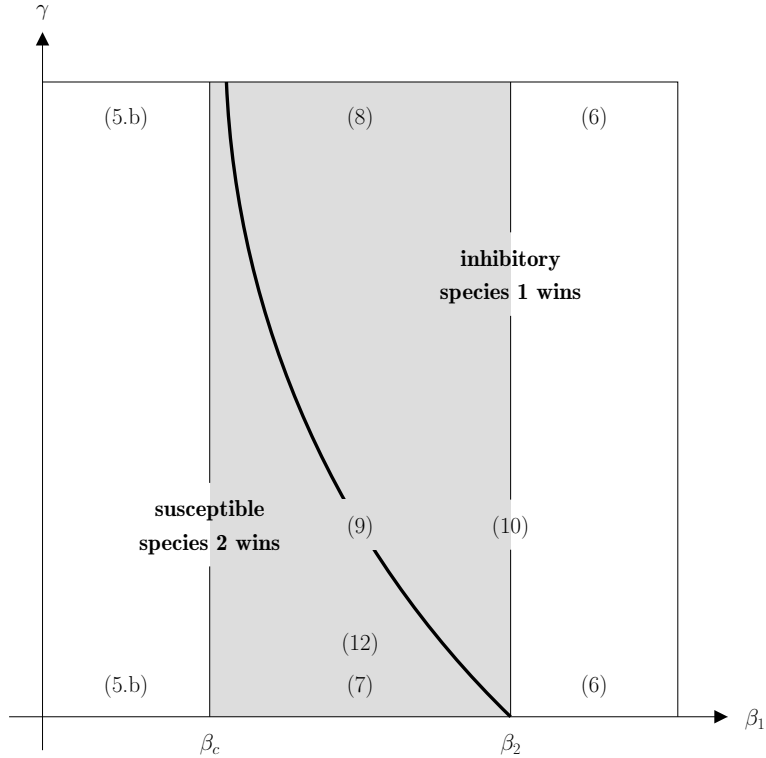}
\caption{\upshape{
 Phase structure of the process for a fixed~$\beta_2 > \beta_c$.}}
\label{fig:PS1}
\end{figure}
\begin{figure}[t!]
\centering
\scalebox{0.35}{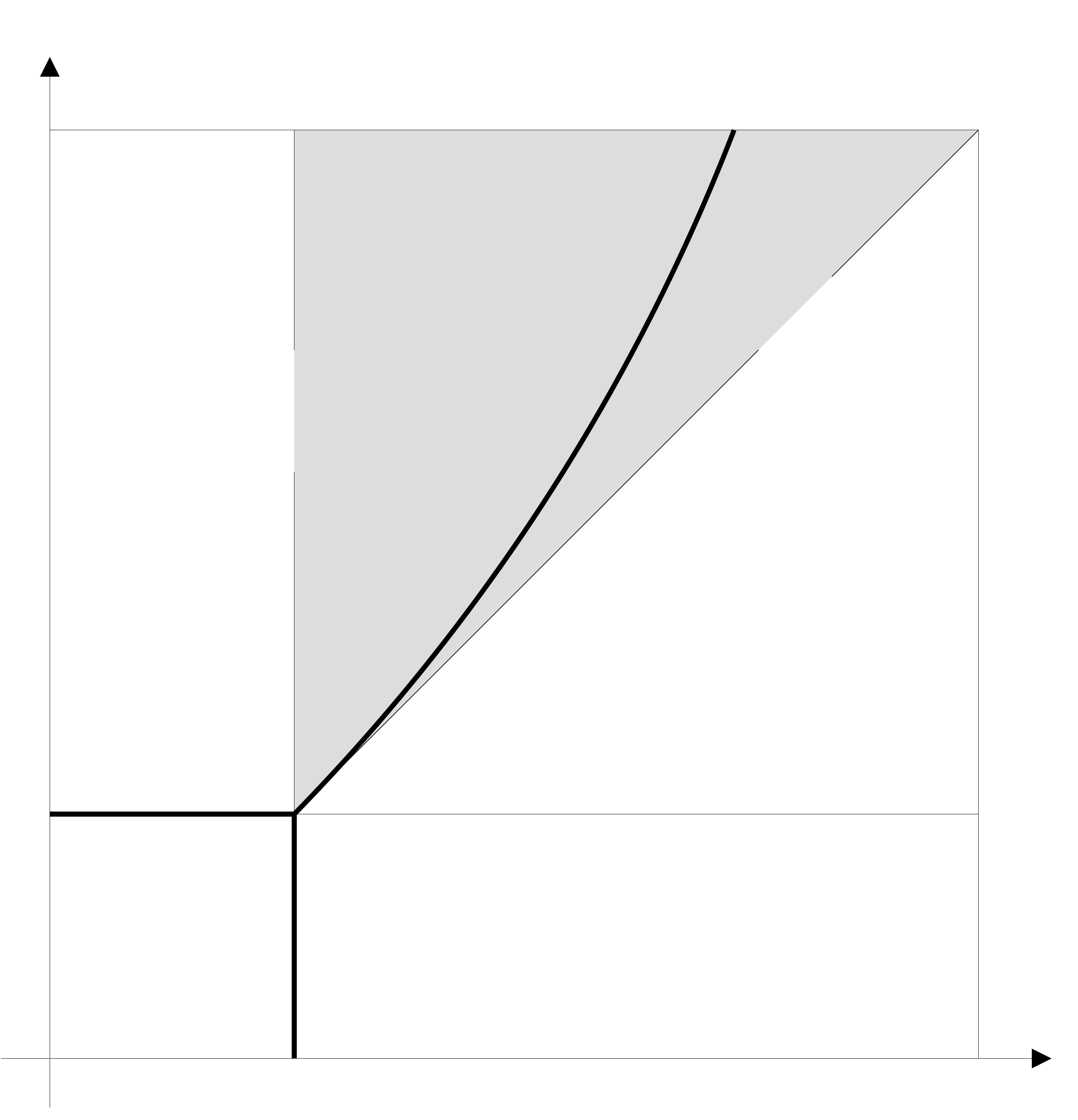}
\caption{\upshape{
 Phase structure of the process for a fixed~$\gamma > 0$.}}
\label{fig:PS2}
\end{figure}
 This, together with the monotonicity with respect to~$\gamma$, implies that the inhibitory species wins whenever it is the fittest:
\begin{equation}
\label{eq:MCP-coupling}
\begin{array}{rcl}
\beta_1 > \beta_2 > \beta_c & \Longrightarrow & \hbox{species 1 wins}. \end{array}
\end{equation}
 Combining~\eqref{eq:extinction}--\eqref{eq:MCP-coupling} gives a complete description of the long-term behavior of the spatial stochastic model outside the parameter region~$\beta_2 \geq \beta_1 > \beta_c$. \\
\indent
 We now look at the behavior of the process in the parameter region~$\beta_2 \geq \beta_1 > \beta_c$, which is represented by the gray rectangle in Figure~\ref{fig:PS1} and the gray triangle in Figure~\ref{fig:PS2}.
 In this case, the long-term behavior also depends on the strength of allelopathy:
 there is a critical value for~$\gamma$ below which the susceptible species wins and above which the inhibitory species survives.
 Neuhauser's results were later improved by Durrett and Neuhauser~\cite{Durrett_Neuhauser_1997} using a block construction to control the rate of invasion of the superior competitor.
 Applying a standard perturbation argument to their construction in the scenario where the susceptible species is the superior competitor implies that the susceptible species still wins when the inhibitory effects are sufficiently weak:
\begin{theorem}
\label{th:MCP-block}
 For all~$\beta_2 > \beta_1 > \beta_c$, there is~$\gamma_- > 0$ such that
\begin{equation}
\label{eq:MCP-block}
\begin{array}{rcl}
\gamma < \gamma_- & \Longrightarrow & \hbox{species 2 wins}. \end{array}
\end{equation}
\end{theorem}
 In contrast, using that the inhibitory species does not feel the presence of the susceptible species in the limiting case~$\gamma = \infty$, together with a block construction and a perturbation argument, shows that the inhibitory species survives when the inhibitory effects are sufficiently strong:
\begin{theorem}
\label{th:CP}
 For all~$\beta_2 > \beta_1 > \beta_c$, there is~$\gamma_+ < \infty$ such that
\begin{equation}
\label{eq:CP}
\begin{array}{rcl}
\gamma > \gamma_+ & \Longrightarrow & \hbox{species 1 survives}. \end{array}
\end{equation}
 Species~1 wins if in addition~$M = d = 1$.
\end{theorem}
 Numerical simulations suggest that clusters quickly form and either the susceptible species or the inhibitory species wins.
 In particular, we conjecture that, like in the mean-field model, coexistence is not possible except maybe for a parameter region with measure zero.
 This, together with the previous two theorems and the monotonicity with respect to~$\gamma$, implies that for all~$\beta_2 > \beta_1 > \beta_c$, there exists a nondegenerate critical value~$\gamma_c$ such that
\begin{equation}
\label{eq:gammac}
\begin{array}{rcl}
\gamma < \gamma_c & \Longrightarrow & \hbox{species 2 wins} \vspace*{4pt} \\
\gamma > \gamma_c & \Longrightarrow & \hbox{species 1 wins}. \end{array}
\end{equation}
 Figure~\ref{fig:IPS} shows realizations of the process near the phase transition, with~$\gamma$ slightly subcritical in the two pictures on the left where the susceptible species wins and~$\gamma$ slightly supercritical in the two pictures on the right where the inhibitory species wins.
 Combining~\eqref{eq:MCP-coupling} and~\eqref{eq:gammac} shows that when the susceptible species is the better competitor there is a phase transition at some nondegenerate critical value~$\gamma_c$ whereas when the inhibitory species is the better competitor the critical value~$\gamma_c$ is degenerate, equal to zero.
 We conjecture that the transition between these two regimes is continuous at point~$\beta_1 = \beta_2$ in the sense that, when the species are equally fit, the inhibitory species wins even in the presence of very weak inhibitory effects, meaning that the critical value~$\gamma_c$ is again degenerate in the symmetric case.
 Using duality techniques inspired from Neuhauser~\cite{Neuhauser_1992}~(tree structure of the dual process, ancestor hierarchy, renewal points, etc.) and the transience of the first ancestor in high dimensions, we were able to prove this conjecture when~$d \geq 3$.
\begin{theorem}
\label{th:MCP-dual}
 Assume that~$\beta_1, \beta_2 > \beta_c$. Then,
\begin{equation}
\label{eq:MCP-dual}
\begin{array}{rcl}
\beta_1 = \beta_2 \quad \hbox{and} \quad \gamma > 0 & \Longrightarrow & \hbox{species 1 wins in~$d \geq 3$}. \end{array}
\end{equation}
\end{theorem}
 We now look at the effects of the birth rate of the susceptible species.
 In the limiting case~$\beta_2 = \infty$, connected components of~2s with at least two individuals can only expand because each time a type~2 individual dies the resulting empty site is instantaneously reinvaded by the offspring of a nearby type~2 individual.
 This implies that, even if the inhibitory effects are strong, the susceptible species still wins provided it is a good enough competitor:
\begin{theorem}
\label{th:strong-2}
 For all~$\beta_1, \gamma > 0$, there is~$\beta_+ < \infty$ such that
\begin{equation}
\label{eq:strong-2}
\begin{array}{rcl}
\beta_2 > \beta_+ & \Longrightarrow & \hbox{species 2 wins}. \end{array}
\end{equation}
\end{theorem}
 The theorem only gives the existence of a finite~(possibly very large) critical value~$\beta_+$.
 Our last result shows that, in the presence of long range interactions, the allelopathic model can be coupled with the grass-bush-tree system~\cite{Durrett_Schinazi_1993,Durrett_Swindle_1991} to identify an explicit parameter region in which the susceptible species survives.
 More precisely, we have the following.
\begin{theorem}
\label{th:GBT}
 For all~$\beta_2 > \beta_1^2 > 1$ and~$\gamma \leq \beta_1$, there is~$M_0 < \infty$ such that
\begin{equation}
\label{eq:GBT}
\begin{array}{rcl}
 M > M_0 & \Longrightarrow & \hbox{species 2 survives}. \end{array}
\end{equation}
\end{theorem}
 The phase structure of the process obtained from~\eqref{eq:extinction}--\eqref{eq:GBT} is depicted in Figures~\ref{fig:PS1} and~\ref{fig:PS2}.
 The gray rectangle in the first picture and the gray triangle in the second picture represent the parameter region covered in our theorems.
 The phase structure does not show any coexistence phase because this is what our numerical simulations suggest.
 However, the absence of a phase of coexistence in the spatial model is still an open problem. \\
\indent
 The rest of this paper is devoted to the proofs and organized as follows.
 Section~\ref{sec:mean-field} gives a complete description of the phase structure of the mean-field model.
 In preparation for the analysis of the spatial model, Section~\ref{sec:GR} explains how to construct the process from a graphical representation and deduce the monotonicity with respect to each of the parameters.
 The proofs of Theorems~\ref{th:MCP-block},~\ref{th:CP},~\ref{th:strong-2}~and~\ref{th:GBT} can be found in Sections~\ref{sec:MCP-CP}--\ref{sec:GBT} and rely on coupling arguments and/or block constructions.
 Finally, the proof of Theorem~\ref{th:MCP-dual}, which is more involved, will be carried out in Section~\ref{sec:MCP-dual}.

%%%%%%%%%%%%%%%%%%%%%%%%%%%%%%%%%%%%%%%%%%%%%%%%%%%%%%%%%%%%%%%%%%%%%%%%%%%%%%%%%%%%%%%%%%%%%%%%%%%%%%%%%%%%%%%%%%%%%%%%%%%%%%%%%%%%%%%%%%%%%%%%%%%%%%%%%%%%%%%%%%%%%%%%%%%%%%%%%%%%%%%%%%%%%%%%%%

\section{Mean-field analysis}
\label{sec:mean-field}
 This section is devoted to the analysis of the mean-field model~\eqref{eq:mean-field}.
 In particular, we give a complete description of the phase structure with the limiting density if each type based on the parameters of the system.
 The key parameter regions that appear in our analysis are
 $$ \begin{array}{rcl}
      B_0 & \n = \n & \{(\beta_1, \beta_2, \gamma) \in \R_+^3: \beta_1 < 1 \ \hbox{and} \ \beta_2 < 1 \} \vspace*{4pt} \\
      B_1 & \n = \n & \{(\beta_1, \beta_2, \gamma) \in \R_+^3 : \beta_1 > 1 \ \hbox{and} \ \beta_2 < (1 + \gamma) \beta_1 - \gamma \} \vspace*{4pt} \\
      B_2 & \n = \n & \{(\beta_1, \beta_2, \gamma) \in \R_+^3 : \beta_2 > 1 \ \hbox{and} \ \beta_2 > \beta_1 \}. \end{array} $$
 Setting the right-hand side of~\eqref{eq:mean-field} equal to zero gives the three trivial fixed points
 $$ p_0 = (0, 0), \quad p_1 = \bigg(1 - \frac{1}{\beta_1}, 0 \bigg), \quad p_2 = \bigg(0, 1 - \frac{1}{\beta_2} \bigg) $$
 and the nontrivial fixed point~$p_{12} = (\bar u_1, \bar u_2)$ where
\begin{equation}
\label{eq:mean-field-1}
\bar u_1 = \frac{1}{\gamma} \bigg(\frac{\beta_2}{\beta_1} - 1 \bigg) \quad \hbox{and} \quad
\bar u_2 = 1 - \frac{1}{\beta_1} - \bar u_1 = 1 - \frac{1}{\beta_1} - \frac{1}{\gamma} \bigg(\frac{\beta_2}{\beta_1} - 1 \bigg).
\end{equation}
 The first step is to identify the sets of parameters for which each of the four fixed points belongs to the biologically relevant two-dimensional simplex~$\Delta^2$.
 The fixed point~$p_0$ is always in the two-dimensional simplex, the fixed point~$p_1$ belongs to the simplex if and only if~$\beta_1 > 1$, and the fixed point~$p_2$ belongs to the simplex if and only if~$\beta_2 > 1$.
 In addition,
\begin{equation}
\label{eq:mean-field-2}
\begin{array}{rclcl}
  p_{12} \in \Delta^2 & \Longleftrightarrow & 0 < \bar u_1, \bar u_2 < 1 \vspace*{4pt} \\
                      & \Longleftrightarrow & \beta_1 < \beta_2 < (1 + \gamma) \beta_1 - \gamma & \Longleftrightarrow & (\beta_1, \beta_2, \gamma) \in B_1 \cap B_2. \end{array}
\end{equation}
 To study the local stability, note that the Jacobian matrix is given by
 $$ J (u_1, u_2) = \left(\begin{array}{cc} \beta_1 (1 - 2u_1 - u_2) - 1 & - \beta_1 u_1 \vspace*{2pt} \\ - (\beta_2 + \gamma) u_2 & \beta_2 (1 - u_1 - 2u_2) - 1 - \gamma u_1 \end{array} \right). $$
 In particular, the Jacobian matrices at~$p_0$, $p_1$ and~$p_2$ are triangular, of the form
 $$ J (p_0) = \left(\begin{array}{cc} \beta_1 - 1 & 0 \vspace*{2pt} \\ 0 & \beta_2 - 1 \end{array} \right) \quad
    J (p_1) = \left(\begin{array}{cc} 1 - \beta_1 & \times \vspace*{2pt} \\ 0 & - \gamma \bar u_2 \end{array} \right) \quad
    J (p_2) = \left(\begin{array}{cc} \beta_1 / \beta_2 - 1 & 0 \vspace*{2pt} \\ \times & 1 - \beta_2 \end{array} \right). $$
 Using that~$\bar u_2 > 0$ if and only if~$\beta_2 < (1 + \gamma) \beta_1 - \gamma$ and looking at the sign of the two eigenvalues on the diagonal of each of the Jacobian matrices, we deduce that
 $$ \begin{array}{rcl}
      p_0 \ \hbox{is locally stable} & \Longleftrightarrow & \beta_1 < 1 \ \hbox{and} \ \beta_2 < 1, \ \hbox{i.e.}, \ (\beta_1, \beta_2, \gamma) \in B_0 \vspace*{4pt} \\
      p_1 \ \hbox{is locally stable} & \Longleftrightarrow & \beta_1 > 1 \ \hbox{and} \ \beta_2 < (1 + \gamma) \beta_1 - \gamma, \ \hbox{i.e.}, \ (\beta_1, \beta_2, \gamma) \in B_1 \vspace*{4pt} \\
      p_2 \ \hbox{is locally stable} & \Longleftrightarrow & \beta_2 > 1 \ \hbox{and} \ \beta_2 > \beta_1, \ \hbox{i.e.}, \ (\beta_1, \beta_2, \gamma) \in B_2. \end{array} $$
 Using numerical simulations, Durrett and Levin noticed that~$p_{12}$ is a saddle, which we now prove rigorously.
 First, notice that, because~$1 - \bar u_1 - \bar u_2 = 1 / \beta_1$,
 $$ \beta_1 (1 - 2 \bar u_1 - \bar u_2) - 1 = \beta_1 (1 - \bar u_1 - \bar u_2) - \beta_1 \bar u_1 - 1 = - \beta_1 \bar u_1 $$
 showing that the two coefficients on the first row of the Jacobian matrix at the fixed point~$p_{12}$ are in fact equal.
 It follows that the expression of the determinant reduces to
 $$ \begin{array}{rcl}
    \de \,(J (p_{12})) & \n = \n & - \beta_1 \bar u_1 (\beta_2 (1 - \bar u_1 - 2 \bar u_2) - 1 - \gamma \bar u_1 + (\beta_2 + \gamma) \bar u_2) \vspace*{4pt} \\
                       & \n = \n & - \beta_1 \bar u_1 (\beta_2 (1 - \bar u_1 - \bar u_2) - 1 - \gamma (\bar u_1 - \bar u_2) \vspace*{4pt} \\
                       & \n = \n & - \beta_1 \bar u_1 (\beta_2 / \beta_1 - \gamma (\bar u_1 - \bar u_2) - 1). \end{array} $$
 Then, replacing~$\bar u_1$ and~$\bar u_2$ by their values in~\eqref{eq:mean-field-1}, we get
 $$ \de \,(J (p_{12})) = - \bar u_1 ((1 + \gamma) \beta_1 - \gamma - \beta_2) < 0 $$
 whenever the equivalent conditions in~\eqref{eq:mean-field-2} hold.
 This shows that the two eigenvalues have opposite signs therefore the interior fixed point is always a saddle.
\begin{figure}[t!]
\centering
\scalebox{0.35}{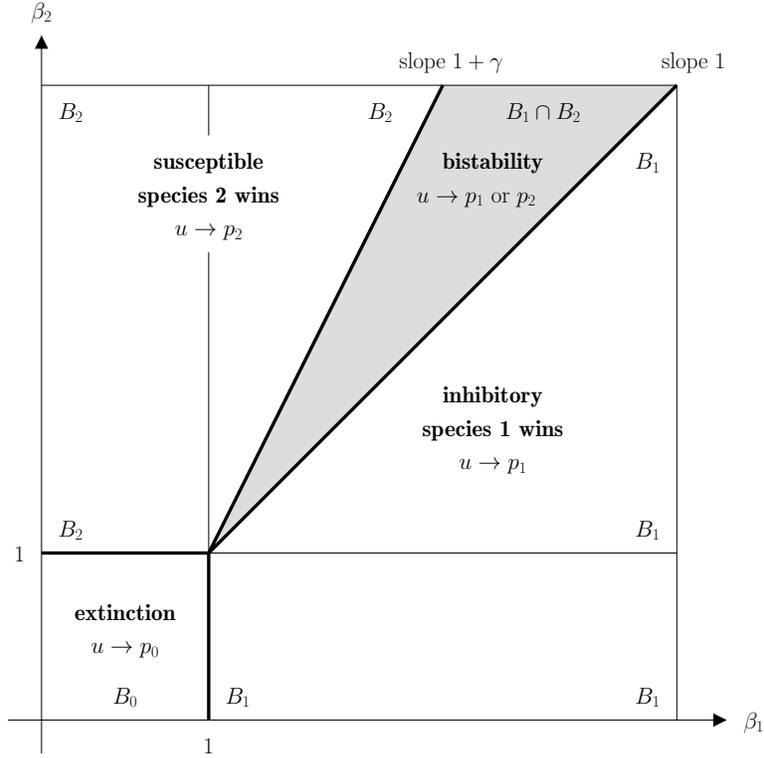}
\caption{\upshape{
 Phase structure of the mean-field model~\eqref{eq:mean-field}.}}
\label{fig:MFPS}
\end{figure}
 To deduce the limiting behavior of the system from the analysis of the local stability of the fixed points, the last step is to apply the~Bendixson--Dulac theorem to exclude the existence of periodic orbits.
 More precisely, the objective is to find a smooth function~$\phi (u_1, u_2)$ such that the sign of~$\nabla \cdot (\phi F_1, \phi F_2)$ is constant almost everywhere in the simplex.
 Using the function~$\phi = 1 / (u_1 u_2)$, we get
 $$ \begin{array}{l}
    \nabla \cdot (\phi F_1, \phi F_2) =
    \displaystyle \frac{\partial (\phi F_1)}{\partial u_1} + \frac{\partial (\phi F_2)}{\partial u_2} \vspace*{4pt} \\ \hspace*{20pt} =
    \displaystyle \frac{\partial}{\partial u_1} \bigg(\frac{\beta_1 (1 - u_1 - u_2) - 1}{u_2} \bigg) +
    \displaystyle \frac{\partial}{\partial u_2} \bigg(\frac{\beta_2 (1 - u_1 - u_2) - 1 - \gamma u_1}{u_1} \bigg) = - \frac{\beta_1}{u_2} - \frac{\beta_2}{u_1} < 0
    \end{array} $$
 for all~$\beta_1, \beta_2 > 0$ and all~$u$ in the interior of the simplex.
 This proves the absence of periodic orbits which, together with the local stability of the fixed points, give the following limiting behavior.
 In the parameter region~$B_0$, the population goes extinct in the sense that~$u \to p_0$.
 In the parameter region~$B_1 \setminus B_2$, the inhibitory species wins in the sense that~$u \to p_1$ when starting with a positive density of~1s.
 Similarly, in the parameter region~$B_2 \setminus B_1$, the susceptible species wins.
 Finally, in the parameter region~$B_1 \cap B_2$, the system is bistable:
 for almost all initial conditions in the simplex, the densities converge to either~$p_1$ or~$p_2$, indicating that the outcome of the competition depends on the initial densities.
 Figure~\ref{fig:MFPS} shows a picture of the phase structure.

%%%%%%%%%%%%%%%%%%%%%%%%%%%%%%%%%%%%%%%%%%%%%%%%%%%%%%%%%%%%%%%%%%%%%%%%%%%%%%%%%%%%%%%%%%%%%%%%%%%%%%%%%%%%%%%%%%%%%%%%%%%%%%%%%%%%%%%%%%%%%%%%%%%%%%%%%%%%%%%%%%%%%%%%%%%%%%%%%%%%%%%%%%%%%%%%%%

\section{Harris' graphical representation}
\label{sec:GR}
 The starting point to study the spatial model and prove our results is to use an idea of Harris~\cite{Harris_1972} to construct the process graphically from a collection of independent Poisson processes/exponential clocks.
 In the case of the allelopathic model, we use four collections corresponding to the following four updates:
 birth of a~1, birth of a~2, natural death, death of a~2 due to inhibitory effects.
 To construct the process, it is convenient to think of the lattice~$\Z^d$ as the vertex set of a directed graph in which there is a directed edge~$\vec{xy}$ if and only if vertices~$x$ and~$y$ are neighbors, and we let~$N$ be the common size of the interaction neighborhoods.
 The exponential clocks are attached to the vertices and directed edges of the graph as follows.
\begin{itemize}
 \item {\bf Inhibition}.
       Place an exponential clock with rate~$\gamma / N$ along each directed edge~$\vec{xy}$.
       Each time the clock rings, say at time~$t$, draw an arrow~$(x, t) \to (y, t)$ labeled with a~0 to indicate that if the tail of the arrow is occupied by a~1 and the head of the arrow is occupied by a~2 then the head of the arrow becomes empty. \vspace*{4pt}
 \item {\bf Birth of a~1}.
       Place an exponential clock with rate~$\beta_1 / N$ along each directed edge~$\vec{xy}$.
       Each time the clock rings, say at time~$t$, draw an arrow~$(x, t) \to (y, t)$ labeled with a~1 to indicate that if the tail of the arrow is occupied by a~1 and the head of the arrow is empty then the head of the arrow becomes occupied by a~1. \vspace*{4pt}
 \item {\bf Birth of a~2}.
       Place an exponential clock with rate~$\beta_2 / N$ along each directed edge~$\vec{xy}$.
       Each time the clock rings, say at time~$t$, draw an arrow~$(x, t) \to (y, t)$ labeled with a~2 to indicate that if the tail of the arrow is occupied by a~2 and the head of the arrow is empty then the head of the arrow becomes occupied by a~2. \vspace*{4pt}
 \item {\bf Natural death}.
       Place an exponential clock with rate one at each vertex~$x$.
       Each time the clock rings, say at time~$t$, put a cross at~$(x, t)$ to indicate that site~$x$ becomes empty.
\end{itemize}
 Figure~\ref{fig:GR} shows an example of realization of this graphical representation and how to construct the process starting from a given initial configuration.
\begin{figure}[t!]
\centering
\scalebox{0.35}{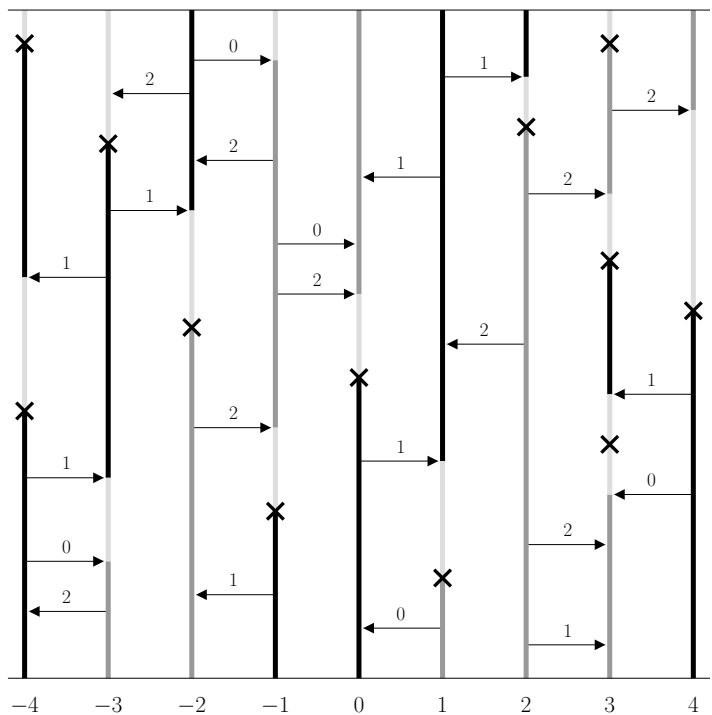}
\caption{\upshape{
 Graphical representation of the allelopathic model.
 The black lines represent the space-time region occupied by the inhibitory species and the gray lines the region occupied by the susceptible species.}}
\label{fig:GR}
\end{figure}
 An argument of Harris~\cite{Harris_1972} based on techniques from percolation theory shows more generally that the allelopathic model on the infinite integer lattice starting from any initial configuration is well-defined and can be constructed using the collection of independent exponential clocks above. \\
\indent
 The graphical representation can also be used together with the superposition properties of the exponential distribution to couple processes with different parameters or different dynamical structures, and deduce some monotonicity results.
 For instance, fix~$\gamma_1 \leq \gamma_2$, and let~$\xi_t^i$ be the process with birth rates~$\beta_1$ and~$\beta_2$, and parameter~$\gamma_i$ for~$i = 1, 2$.
 Think of the first process as being generated from the graphical representation described above with~$\gamma = \gamma_1$ while, according to the superposition property, the second process can be constructed from the same graphical representation as the first process by also adding type~0 arrows along the directed edges at rate~$(\gamma_2 - \gamma_1) / N$, which results in a joint construction~(coupling) of the two processes.
 In addition, starting both processes from the same initial configuration, one can easily check that
 $$ \begin{array}{rclcrcl} \xi_t^1 (x) = 1 & \Longrightarrow & \xi_t^2 (x) = 1 & \quad \hbox{and} \quad & \xi_t^2 (x) = 2 & \Longrightarrow & \xi_t^1 (x) = 2. \end{array} $$
 Using also the monotonicity of the expectation, we deduce that, the birth rates being fixed, survival of the inhibitory species when~$\gamma = \gamma_1$ implies survival of the inhibitory species when~$\gamma = \gamma_2$, and survival of the susceptible species when~$\gamma = \gamma_2$ implies survival of the susceptible species when~$\gamma = \gamma_1$.
 The monotonicity with respect to each of the birth rates mentioned in the introduction can be proved similarly by coupling the birth arrows. \\
\indent
 Recalling that the multitype contact process corresponds to the special case~$\gamma = 0$, the previous monotonicity result also implies that if the~1s win in the multitype contact process then they also win in the allelopathic model with the same birth parameters for all~$\gamma \geq 0$.
 This, together with~Neuhauser's result~\cite[Theorem~1]{Neuhauser_1992} which states that
 $$ \begin{array}{rcl} \beta_1 > \beta_2 > \beta_c & \Longrightarrow & \hbox{species 1 wins} \end{array} $$
 for the multitype contact process, implies that the same result~\eqref{eq:MCP-coupling} also holds for the allelopathic model.
 The graphical representation will also be used in the next sections to couple the process with the grass-bush-tree system~\cite{Durrett_Schinazi_1993,Durrett_Swindle_1991}, define good events in certain space-time regions when using block constructions, and study the dual process of the model.

%%%%%%%%%%%%%%%%%%%%%%%%%%%%%%%%%%%%%%%%%%%%%%%%%%%%%%%%%%%%%%%%%%%%%%%%%%%%%%%%%%%%%%%%%%%%%%%%%%%%%%%%%%%%%%%%%%%%%%%%%%%%%%%%%%%%%%%%%%%%%%%%%%%%%%%%%%%%%%%%%%%%%%%%%%%%%%%%%%%%%%%%%%%%%%%%%%

\section{Proofs of Theorems~\ref{th:MCP-block} and~\ref{th:CP}}
\label{sec:MCP-CP}
 This section is devoted to the proofs of Theorem~\ref{th:MCP-block} and~\ref{th:CP}.
 Assuming that the susceptible species is a better competitor than the inhibitory species, recall that the theorems state that the susceptible species wins if the inhibitory effects are sufficiently weak whereas the inhibitory species survives if the inhibitory effects are sufficiently strong.
 The proofs of both theorems consist in comparing the process in the two extreme cases~$\gamma = 0$ and~$\gamma = \infty$ with well-known interacting particle systems, and then using a block construction along with a perturbation argument to prove the existence of a phase transition at a nondegenerate critical value~$\gamma_c$.
 Throughout this section, $P_{\gamma}$ will refer to the probability for the process with parameter~$\gamma$. \\
\indent
 The technique now known as the block construction appeared for the first time in the work of Bramson and Durrett~\cite{Bramson_Durrett_1988} and is reviewed in detail in Durrett's lecture notes~\cite{Durrett_1995}.
 The general idea of this approach is to couple the process properly rescaled in space and time with oriented site percolation, which we now briefly describe.
 To begin with, let
 $$ \Lat = \{(m, n) \in \Z^d \times \Z_+ : m_1 + \cdots + m_d + n \ \hbox{is even} \} $$
 which we turn into a directed graph by placing arrows
 $$ (m, n) \to (m', n') \quad \hbox{if and only if} \quad |m_1 - m_1'| + \cdots + |m_d - m_d'| = 1 \ \ \hbox{and} \ \ n' = n + 1. $$
 The sites in~$\Lat$ are assumed to be open or closed with probability~$p$ and~$1 - p$, respectively, and whether a finite collection of sites are open or closed are assumed to be independent events when these sites are sufficiently far apart.
 We say that a site is wet if it can be reached from a directed path of open sites starting at level~$n = 0$, we let~$\C_0$ be the cluster of sites that can be reached from a directed path of open sites starting at the origin, and we call the event that this cluster is infinite the percolation event.
 Using a so-called contour argument, it can be proved that there exists a nondegenerate critical value~$p_c \in (0, 1)$ for the density of open sites such that
 $$ \begin{array}{rcl} p > p_c & \Longrightarrow & P_p (|\C_0| = \infty) > 0, \end{array} $$
 and we refer to Durrett~\cite[Section~10]{Durrett_1984} for the details of the proof.
 A consequence of this result is that, in the supercritical phase~$p > p_c$ and starting with infinitely many open sites at level~$n = 0$, the density of wet sites at level~$n$ converges to a positive limit as~$n \to \infty$. \\
\indent
 To prove Theorem~\ref{th:MCP-block}, the basic idea is to observe that, under the assumptions of the theorem, the allelopathic model consists of a small perturbation of a multitype contact process in which type~2 individuals win.
 To turn our intuition into a rigorous proof, we will apply a perturbation argument to the block construction Durrett and Neuhauser~\cite[Section~3]{Durrett_Neuhauser_1997} used to study the multitype contact process, the particular case~$\gamma = 0$.
 Changing a little bit the notation in their construction, for each site~$(m, n) \in \Lat$, we define the space-time boxes
 $$ \begin{array}{rcl}
      A_{m, n} & \n = \n & (mL, nT) + ([- L, L]^d \times \{T \}) \vspace*{4pt} \\
      B_{m, n} & \n = \n & (mL, nT) + ([- L, L]^d \times [T, 2T]) \vspace*{4pt} \\
      C_{m, n} & \n = \n & (mL, nT) + ([- 3L, 3L]^d \times [0, 2T]) \quad \hbox{where} \quad T = L^2. \end{array} $$
 Note that~$A_{m, n}$ is flat and represents the bottom of the space-time box~$B_{m, n}$.
 We partition~$A_{m, n}$ into small cubes of size~$L^{0.1} \times \cdots \times L^{0.1}$, and define the events
 $$ \begin{array}{rcl}
     E_{m, n} & \n = \n & \hbox{each of the small cubes in~$A_{m, n}$ contains at least one type~2 individual} \vspace*{4pt} \\
     F_{m, n} & \n = \n & \hbox{the space-time box~$B_{m, n}$ does not contain any type~1 individual}. \end{array} $$
 Durrett and Neuhauser's proof, which relies on duality techniques along with a repositioning algorithm, implies that, for the multitype contact process with~$\beta_2 > \beta_1 > \beta_c$ and for all~$\ep > 0$, there exists a collection of good events~$G_{m, n}$ that only depend on the graphical representation in the larger space-time box~$C_{m, n}$ such that, for all~$L$ sufficiently large,
 $$ \begin{array}{rl}
     \hbox{(a)} & P_0 (G_{m, n}) \geq 1 - \ep / 2 \vspace*{4pt} \\
     \hbox{(b)} & E_{m, n} \cap G_{m, n} \ \Longrightarrow \ E_{m', n'} \cap F_{m', n'} \ \ \hbox{for all} \ \ (m', n') \leftarrow (m, n). \end{array} $$
 In words, if the bottom of a space-time box contains many~2s then, with high probability, not only the bottom of the boxes immediately above contains many~2s but also these boxes do not contain any~1s, as shown in Figure~\ref{fig:MCP-block} in the one-dimensional case.
\begin{figure}[t!]
\centering
\scalebox{0.35}{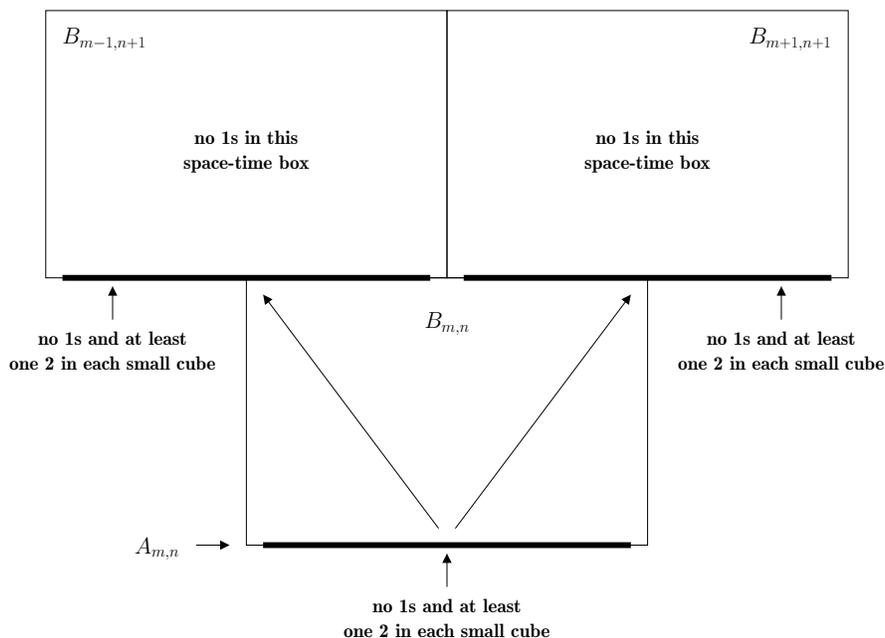}
\caption{\upshape{
 Illustration of the block construction in Durrett and Neuhauser~\cite{Durrett_Neuhauser_1997}.}}
\label{fig:MCP-block}
\end{figure}
 In order to apply a perturbation argument, we now think of the allelopathic model as being generated from the same graphical representation as the multitype contact process along with type~0 arrows starting from each site at rate~$\gamma$, as shown in Figure~\ref{fig:GR}.
 The scale parameter~$L$ being fixed, there exists~$\gamma_- > 0$ small such that the probability that the graphical representation of the allelopathic model coincides with that of the multitype contact process in the box~$C_{m, n}$ is given by
 $$ \begin{array}{l}
      P_{\gamma} (\hbox{no type~0 arrows in the space-time box} \ C_{m, n}) \vspace*{4pt} \\ \hspace*{40pt} =
      P (\poisson (\gamma \vol (C_{m, n})) = 0) = \exp (- 2L^2 (6L + 1)^d \,\gamma) \geq 1 - \ep / 2 \end{array} $$
 for all~$\gamma < \gamma_-$.
 In particular, for all~$\gamma < \gamma_-$,
 $$ P_{\gamma} (G_{m, n}) = 1 - P_{\gamma} (G_{m, n}^c) \geq 1 - P_0 (G_{m, n}^c) - \ep / 2 \geq 1 - \ep / 2 - \ep / 2 = 1 - \ep. $$
 Calling site~$(m, n) \in \Lat$ a good site whenever~$E_{m, n} \cap F_{m, n}$ occurs,~\cite[Theorem~A.4]{Durrett_1995} implies that there is a coupling of the process and oriented site percolation with parameter~$p = 1 - \ep$ such that the set of good sites in the interacting particle system dominates the set of wet sites in the percolation model.
 Choosing~$\ep < 1 - p_c$ then implies that the density of good sites converges to a positive limit as~$n \to \infty$, which proves survival of the susceptible species.
 This does not fully prove the theorem because there is still a positive density~$\ep > 0$ of closed sites, and the corresponding space-time box~$B_{m, n}$ can potentially contain~1s.
 To prove extinction of the inhibitory species, the last step is to use a result due to Durrett~\cite{Durrett_1992} which shows that, when~$\ep$ is small enough, not only the open sites percolate but also the closed sites do not percolate.
 Because the~1s cannot appear spontaneously, the presence of an individual of type~1 in box~$B_{m, n}$ implies the existence of a directed path of closed sites from level zero to site~$(m, n)$, an event whose probability decays exponentially fast with~$n$ due to the lack of percolation of the closed sites.
 This shows extinction of the inhibitory species and completes the proof of Theorem~\ref{th:MCP-block}. \\
\indent
 To prove Theorem~\ref{th:CP}, we observe that, under the assumptions of the theorem, the set of~1s in the allelopathic model consists of a small perturbation of a supercritical contact process.
 The first work showing that the supercritical contact process properly rescaled in space and time dominates oriented site percolation with parameter arbitrarily close to one is the paper of Bezuidenhout and Grimmett~\cite{Bezuidenhout_Grimmett_1990} where they used this coupling together with a perturbation argument to prove extinction of the critical contact process.
 Instead of explaining their construction, we simply reuse the result of Durrett and Neuhauser~\cite{Durrett_Neuhauser_1997} exchanging the roles of the two types of particles, in the absence of~2s.
 Using the same space-time boxes but redefining the events
 $$ \begin{array}{rcl} E_{m, n} & \n = \n & \hbox{each of the small cubes in~$A_{m, n}$ contains at least one type~1 individual}, \end{array} $$
 their result implies that, for the~(supercritical) contact process with~$\beta_1 > \beta_c$ and for all~$\ep > 0$, there exists a collection of good events~$G_{m, n}$ that only depend on the graphical representation in the larger space-time box~$C_{m, n}$ such that, for all~$L$ sufficiently large,
\begin{equation}
\label{eq:CP-1}
\begin{array}{rl}
\hbox{(a)} & P_{\infty} (G_{m, n}) \geq 1 - \ep / 2 \vspace*{4pt} \\
\hbox{(b)} & E_{m, n} \cap G_{m, n} \ \Longrightarrow \ E_{m', n'} \ \ \hbox{for all} \ \ (m', n') \leftarrow (m, n). \end{array}
\end{equation}
 In words, if the bottom of a space-time box~$B_{m, n}$ contains many~1s then, with high probability, the bottom of the boxes immediately above also contains many~1s.
 Returning to the allelopathic model, observe that, in the limiting case~$\gamma = \infty$, the~2s in the neighborhood of a~1 are instantaneously killed, which implies that the~1s only attempt to give birth onto empty sites or sites already occupied by a type~1 particle.
 This shows that the~1s do not feel the presence of the~2s and therefore evolve according to a supercritical contact process.
 The next step is to apply a perturbation argument to prove that the good events~$G_{m, n}$ above still occur with probability arbitrarily close to one when~$\gamma$ is large but finite, which is done in the following lemma.
\begin{lemma}
\label{lem:CP}
 For all~$\ep > 0$, there exist good events~$\bar G_{m, n}$ and large~$L$ and~$\gamma_+$ such that
 $$ \begin{array}{rl}
    \hbox{(a)} & P_{\gamma} (\bar G_{m, n}) \geq 1 - \ep \ \hbox{for all} \ \ \gamma > \gamma_+ \vspace*{4pt} \\
    \hbox{(b)} & E_{m, n} \cap \bar G_{m, n} \ \Longrightarrow \ E_{m', n'} \ \ \hbox{for all} \ \ (m', n') \leftarrow (m, n). \end{array} $$
\end{lemma}
\begin{proof}
 Think of the allelopathic model as being generated from the same graphical representation as the contact process along with type~0 arrows starting from each site at rate~$\gamma$ and type~2 arrows starting from each site at rate~$\beta_2$, as shown in Figure~\ref{fig:GR}.
 Then, the set of~1s in the space-time box~$C_{m, n}$ evolves as if there were no~2s if each time a~2 gives birth onto a site~$x$ next to a~1, the offspring dies before any of the surrounding~1s tries to send an offspring to~$x$.
 This happens in particular if each time there is a type~2 arrow pointing at~$(x, t) \in C_{m, n}$, there are~$N$ type~0 arrows starting from each site in~$N_x$ and pointing at~$x$ before any type~1 arrow pointing at~$x$.
 Calling this event~$F_{m, n}$, part~(b) of the lemma holds for the events
 $$ \bar G_{m, n} = F_{m, n} \cap G_{m, n} \quad \hbox{for all} \quad (m, n) \in \Lat. $$
 To estimate the probability of~$\bar G_{m, n}$, define the events
 $$ H_{m, n} = \hbox{there are less than~$2h_L = 4 \beta_2 L^2 (6L + 1)^d$ type~2 arrows pointing at~$C_{m, n}$}. $$
 Because type~2 arrows occur at rate~$\beta_2$ at each site,
\begin{equation}
\label{eq:CP-2}
  P (H_{m, n}^c) = P (\poisson (\beta_2 \vol (C_{m, n})) \geq 2h_L) = P (\poisson (h_L) \geq 2h_L)) \leq \ep / 4
\end{equation}
 for all~$L$ sufficiently large.
 Now, let
 $$ X_1, X_2, \ldots, X_N = \exponential (\gamma / N) \quad \hbox{and} \quad Y = \exponential (\beta_1) $$
 be independent.
 The parameter~$L$ being fixed so that~\eqref{eq:CP-1} and~\eqref{eq:CP-2} hold, recalling the rate of the type~0 arrows and type~1 arrows, and using the superposition property, we get
 $$ \begin{array}{rcl}
      P_{\gamma} (F_{m, n}^c \,| \,H_{m, n}) & \n \leq \n & 2h_L P_{\gamma} (\max (X_1, X_2, \ldots, X_N) > Y) \vspace*{4pt} \\
                                             & \n \leq \n & 2N h_L P_{\gamma} (X_1 > Y) = 2N h_L \beta_1 / (\beta_1 + \gamma / N) \leq \ep / 4 \end{array} $$
 for all~$\gamma$ sufficiently large.
 In particular, there exists~$\gamma_+ < \infty$ such that
\begin{equation}
\label{eq:CP-3}
\begin{array}{rcl}
  P_{\gamma} (F_{m, n}^c) & \n = \n & P_{\gamma} (F_{m, n}^c \,| \,H_{m, n}) P (H_{m, n}) + P_{\gamma} (F_{m, n}^c \,| \,H_{m, n}^c) P (H_{m, n}^c) \vspace*{4pt} \\
                          & \n \leq \n & P_{\gamma} (F_{m, n}^c \,| \,H_{m, n}) + P (H_{m, n}^c) \leq \ep / 4 + \ep / 4 = \ep / 2 \end{array}
\end{equation}
 for all~$\gamma > \gamma_+$.
 Combining~\eqref{eq:CP-1} and~\eqref{eq:CP-3}, we conclude that
 $$ P_{\gamma} (\bar G_{m, n}) \geq 1 - P_{\gamma} (F_{m, n}^c) - P_{\infty} (G_{m, n}^c) \geq 1 - \ep / 2 - \ep / 2 = 1 - \ep $$
 for all~$\gamma > \gamma_+$, which proves part~(a) of the lemma.
\end{proof} \\ \\
 As previously, calling site~$(m, n) \in \Lat$ a good site when~$E_{m, n}$ occurs, it follows from the lemma that there is a coupling with oriented site percolation with parameter~$p = 1 - \ep$ such that the set of good sites dominates the set of wet sites in the percolation model.
 Choosing~$\ep$ small enough implies that the density of good sites converges to a positive limit as~$n \to \infty$, which shows survival of the inhibitory species.
 The stochastic domination also implies that there is a site~(in fact infinitely many sites)~$(m, 0)$ such that the cluster of good sites starting at~$(m, 0)$ expands linearly in all directions.
 Because invasion paths cannot jump over or go around each other in the presence of one-dimensional nearest neighbor interactions, we deduce that this cluster is void of~2s therefore species~1 wins when~$M = d = 1$.
 This completes the proof of Theorem~\ref{th:CP}.

%%%%%%%%%%%%%%%%%%%%%%%%%%%%%%%%%%%%%%%%%%%%%%%%%%%%%%%%%%%%%%%%%%%%%%%%%%%%%%%%%%%%%%%%%%%%%%%%%%%%%%%%%%%%%%%%%%%%%%%%%%%%%%%%%%%%%%%%%%%%%%%%%%%%%%%%%%%%%%%%%%%%%%%%%%%%%%%%%%%%%%%%%%%%%%%%%%

\section{Proof of Theorem~\ref{th:strong-2}}
\label{sec:strong-2}
 This section focuses on the effects of the competitiveness of the susceptible species, measured by the birth rate~$\beta_2$.
 Recall from~\eqref{eq:MCP-coupling} that the inhibitory species wins whenever it is a better competitor than the susceptible species.
 In contrast, Theorem~\ref{th:strong-2} states that, even when the inhibitory effects are strong, the susceptible species wins if it is highly competitive.
 This and monotonicity imply the existence of a unique phase transition from extinction to survival of the susceptible species in the direction of the parameter~$\beta_2$ as shown in Figure~\ref{fig:PS2}.
 To prove the theorem, let
 $$ \Lambda_- = \{0, 1 \}^d \quad \hbox{and} \quad \Lambda_+ = \{-1, 0, 1, 2 \}^d $$
 and let~$\bar \xi_t$ be the allelopathic model starting from
 $$ \bar \xi_0 (x) = 2 \ \ \hbox{for all} \ \ x \in \Lambda_- \quad \hbox{and} \quad \bar \xi_0 (x) \neq 2 \ \ \hbox{for all} \ \ x \notin \Lambda_+ $$
 and modified so that births of~2s outside the larger cube~$\Lambda_+$ are suppressed.
 The next lemma shows that, in the limiting case~$\beta_2 = \infty$, the set of~2s fully invades~$\Lambda_+$ in a finite deterministic time with probability arbitrarily close to one.
\begin{lemma}
\label{lem:infinity}
 Let~$\beta_2 = \infty$.
 Then, for all~$\ep > 0$, there exists~$T < \infty$ such that
 $$ P_{\infty} (\bar \xi_T (x) = 2 \ \hbox{for all} \ x \in \Lambda_+) \geq 1 - \ep / 2. $$
\end{lemma}
\begin{proof}
 Each time a~2 dies at some site~$x \in \Lambda_+$, there is a~2 at some site in~$N_x$ that can send instantaneously an offspring to site~$x$, which shows that the set of~2s can only increase.
 To prove full invasion of the~2s, we need to deal with two problems:
\begin{itemize}
 \item potential~1s in the set~$\Lambda_+ \setminus \Lambda_-$ that block the~2s, \vspace*{4pt}
 \item the fact that particles in~$\Lambda_-$ might not be able to send their offspring to a corner of~$\Lambda_+$, which is the case for instance in the presence of nearest neighbor interactions.
\end{itemize}
 However, regardless of the range, there exist~$x_1, x_2, \ldots, x_K$ with~$K = 4^d - 2^d$ such that
 $$ \Lambda_+ \setminus \Lambda_- = \{x_1, x_2, \ldots, x_K \} \quad \hbox{and} \quad (\Lambda_- \cup \{x_1, \ldots, x_{i - 1}) \cap N_{x_i} \neq \varnothing. $$
 In particular, the time it takes for the~2s to invade~$\Lambda_+$ is smaller than the time~$S$ it takes to have one death mark at~$x_1$ then one death mark at~$x_2$, and so on.
 By memoryless, this is the sum of~$K$ independent exponential random variables~$X_i$ with mean one therefore, for all~$\ep > 0$,
 $$ \begin{array}{l}
      P_{\infty} (\bar \xi_T (x) = 2 \ \hbox{for all} \ x \in \Lambda_+) \geq
      P (S < T) \vspace*{4pt} \\ \hspace*{40pt} \geq
      P (X_i < T / K \ \hbox{for all} \ i = 1, 2, \ldots, K) = (1 - e^{- T/K})^K \geq 1 - \ep / 2 \end{array} $$
 for all~$T$ sufficiently large.
 This completes the proof.
\end{proof} \\ \\
 Applying a perturbation argument like in the proof of Lemma~\ref{lem:CP}, we deduce that for all~$\ep > 0$, there exist sufficiently large~$T < \infty$ and~$\beta_+ < \infty$ such that
\begin{equation}
\label{eq:perturbation}
  P_{\beta_2} (\bar \xi_T (x) = 2 \ \hbox{for all} \ x \in \Lambda_+) \geq 1 - \ep \quad \hbox{for all} \quad \beta_2 > \beta_+.
\end{equation}
 Because the set of~2s in the modified process~$\bar \xi_t$ is dominated by its counterpart in the original allelopathic model with the same parameters and because the evolution rules of the allelopathic model are translation invariant in space, it follows from~\eqref{eq:perturbation} that the set of~2s in the process~$\xi_t$ properly rescaled in space and time dominates supercritical oriented site percolation.
 More precisely, call site~$(m, n) \in \Lat$ a good site whenever the following two events occur:
 $$ \begin{array}{rcl}
     E_{m, n} & \n = \n & \hbox{the box $m + \Lambda_-$ is fully occupied by 2s at time~$nT$} \vspace*{4pt} \\
     F_{m, n} & \n = \n & \hbox{the space-time box~$(m, nT) + (\Lambda_- \times [0, T])$ does not contain any 1s}. \end{array} $$
 Then, it follows from~\eqref{eq:perturbation} that for all~$\beta_1$ and~$\gamma$, and for all~$\ep > 0$, there exists a collection of good events~$G_{m, n}$ such that, for all~$T$ and~$\beta_+$ large,
 $$ \begin{array}{rl}
     \hbox{(a)} & P_{\beta_2} (G_{m, n}) \geq 1 - \ep \vspace*{4pt} \\
     \hbox{(b)} & E_{m, n} \cap G_{m, n} \ \Longrightarrow \ E_{m', n'} \cap F_{m', n'} \ \ \hbox{for all} \ \ (m', n') \leftarrow (m, n) \end{array} $$
 for all~$\beta_2 > \beta_+$.
 In addition, because Lemma~\ref{lem:infinity} applies to the modified process~$\bar \xi_t$, the event~$G_{m, n}$ can be made measurable with respect to the graphical representation of the process in
 $$ (m, nT) + ([- M - 1, M + 2]^d \times [0, 2T]). $$ 
 As previously, this shows that the set of good sites dominates the set of wet sites in an oriented site percolation process with parameter~$p = 1 - \ep$.
 Choosing~$\ep > 0$ small enough to ensure percolation of the open sites and the lack of percolation of the closed sites, we conclude that the susceptible species survives whereas the inhibitory species goes extinct.

%%%%%%%%%%%%%%%%%%%%%%%%%%%%%%%%%%%%%%%%%%%%%%%%%%%%%%%%%%%%%%%%%%%%%%%%%%%%%%%%%%%%%%%%%%%%%%%%%%%%%%%%%%%%%%%%%%%%%%%%%%%%%%%%%%%%%%%%%%%%%%%%%%%%%%%%%%%%%%%%%%%%%%%%%%%%%%%%%%%%%%%%%%%%%%%%%%

\section{Proof of Theorem~\ref{th:GBT}}
\label{sec:GBT}
 This section gives a coupling of the allelopathic model and the grass-bush-tree system introduced by Durrett and Swindle~\cite{Durrett_Swindle_1991} from which Theorem~\ref{th:GBT} follows easily.
 Changing their labeling of the states to facilitate the comparison with the allelopathic model, the sites in the grass-bush-tree system can be in state~0~=~grass, 1~=~tree, and~2~=~bush.
 The process, denoted by~$\zeta_t$, is a variant of the multitype contact process modeling ecological successions in which the tree is a superior competitor that can give birth onto both grass and bushes, so the transition rates are
 $$ \begin{array}{rclcrcl}
      0, 2 \to 1 & \hbox{at rate} & \beta_1 f_1 (x, \zeta) & \qquad & 1 \to 0 & \hbox{at rate} & 1 \vspace*{4pt} \\
         0 \to 2 & \hbox{at rate} & \beta_2 f_2 (x, \zeta) & \qquad & 2 \to 0 & \hbox{at rate} & 1. \end{array} $$
 Note that the dynamical structure of the process is the same as that of the allelopathic model except that the death rate of the~2s is no longer density dependent, which favors the~2s, but the~1s can give birth onto the~2s, which favors the~1s.
 Assuming that~$\gamma \leq \beta_1$, the two processes can be constructed from the same graphical representation to define a coupling~$(\xi_t, \zeta_t)$ as follows.
\begin{itemize}
 \item Place an exponential clock with rate~$\gamma / N$ along each directed edge~$\vec{xy}$.
       Each time the clock rings, say at time~$t$, draw an arrow~$(x, t) \to (y, t)$ labeled with a~0.
       For the allelopathic model~(first coordinate), if the tail is occupied by a~1 and the head is in state~0 then the~1 gives birth through the arrow while if the tail is occupied by a~1 and the head is in state~2 then the~2 dies.
       For the grass-bush-tree system~(second coordinate), if the tail is occupied by a~1 and the head is in state~0 or~2 then the~1 gives birth through the arrow. \vspace*{4pt}
 \item Place an exponential clock with rate~$(\beta_1 - \gamma) / N$ along each directed edge~$\vec{xy}$.
       Each time the clock rings, say at time~$t$, draw an arrow~$(x, t) \to (y, t)$ labeled with a~1.
       For the allelopathic model, if the tail is occupied by a~1 and the head is in state~0 then the~1 gives birth through the arrow.
       For the grass-bush-tree system, if the tail is occupied by a~1 and the head is in state~0 or~2 then the~1 gives birth through the arrow. \vspace*{4pt}
 \item Place an exponential clock with rate~$\beta_2 / N$ along each directed edge~$\vec{xy}$.
       Each time the clock rings, say at time~$t$, draw an arrow~$(x, t) \to (y, t)$ labeled with a~2.
       If the tail is occupied by a~2 and the head is in state~0 then, for both the allelopathic model and the grass-bush-tree system, the~2 gives birth through the arrow. \vspace*{4pt}
 \item Finally, place an exponential clock with rate one at each vertex~$x$.
       Each time the clock rings, say at time~$t$, put a cross at~$(x, t)$ to indicate that, for both the allelopathic model and the grass-bush-tree system, site~$x$ becomes empty.
\end{itemize}
 Note that type~0 arrows create more~0s in the allelopathic model and more~1s in the grass-bush-tree system.
 In particular, starting both processes from the same initial configuration and using the graphical representation above, one expects that, for all~$(x, t) \in \Z^d \times \R_+$,
\begin{equation}
\label{eq:ordering}
\xi_t (x) = 1 \ \Longrightarrow \ \zeta_t (x) = 1 \qquad \hbox{and} \qquad \xi_t (x) = 2 \ \Longleftarrow \ \zeta_t (x) = 2.
\end{equation}
 To prove this result, it suffices to prove that the set of states
 $$ S = \{(0, 0), (0, 1), (1, 1), (2, 0), (2, 1), (2, 2) \} $$
 is closed under the dynamics of the coupling~$(\xi_t, \zeta_t)$.
 Regardless of the state of the coupling, a cross will turn this state into~$(0, 0) \in S$ so the crosses do not create new states not in~$S$.
 Similarly, checking the effect of each type of arrow on all possible configurations, we get
 $$ \begin{array}{cccccccccccccc}
         &                                                 & \hbox{before} & \hbox{after} & \quad &
         &                                                 & \hbox{before} & \hbox{after} & \quad &
         &                                                 & \hbox{before} & \hbox{after} \\         
  (0, 1) & \n \overset{\hbox{\tiny 0}}{\longrightarrow} \n & (0, 0)        & (0, 1) &&
  (0, 1) & \n \overset{\hbox{\tiny 1}}{\longrightarrow} \n & (0, 0)        & (0, 1) &&
  (2, 0) & \n \overset{\hbox{\tiny 2}}{\longrightarrow} \n & (0, 0)        & (2, 0) \\
  (0, 1) & \n \overset{\hbox{\tiny 0}}{\longrightarrow} \n & (2, 0)        & (2, 1) &&
  (0, 1) & \n \overset{\hbox{\tiny 1}}{\longrightarrow} \n & (2, 0)        & (2, 1) &&
  (2, 0) & \n \overset{\hbox{\tiny 2}}{\longrightarrow} \n & (0, 1)        & (2, 1) \\
  (0, 1) & \n \overset{\hbox{\tiny 0}}{\longrightarrow} \n & (2, 2)        & (2, 1) &&
  (0, 1) & \n \overset{\hbox{\tiny 1}}{\longrightarrow} \n & (2, 2)        & (2, 1) &&
  (2, 1) & \n \overset{\hbox{\tiny 2}}{\longrightarrow} \n & (0, 0)        & (2, 0) \\
  (1, 1) & \n \overset{\hbox{\tiny 0}}{\longrightarrow} \n & (0, 0)        & (1, 1) &&
  (1, 1) & \n \overset{\hbox{\tiny 1}}{\longrightarrow} \n & (0, 0)        & (1, 1) &&
  (2, 1) & \n \overset{\hbox{\tiny 2}}{\longrightarrow} \n & (0, 1)        & (2, 1) \\
  (1, 1) & \n \overset{\hbox{\tiny 0}}{\longrightarrow} \n & (0, 1)        & (1, 1) &&
  (1, 1) & \n \overset{\hbox{\tiny 1}}{\longrightarrow} \n & (0, 1)        & (1, 1) &&
  (2, 2) & \n \overset{\hbox{\tiny 2}}{\longrightarrow} \n & (0, 0)        & (2, 2) \\
  (1, 1) & \n \overset{\hbox{\tiny 0}}{\longrightarrow} \n & (2, 0)        & (0, 1) &&
  (1, 1) & \n \overset{\hbox{\tiny 1}}{\longrightarrow} \n & (2, 0)        & (2, 1) &&
  (2, 2) & \n \overset{\hbox{\tiny 2}}{\longrightarrow} \n & (0, 1)        & (2, 1) \\
  (1, 1) & \n \overset{\hbox{\tiny 0}}{\longrightarrow} \n & (2, 1)        & (0, 1) &&
  (1, 1) & \n \overset{\hbox{\tiny 1}}{\longrightarrow} \n & (2, 2)        & (2, 1) &&
  (2, 2) & \n \overset{\hbox{\tiny 2}}{\longrightarrow} \n & (2, 0)        & (2, 2) \\
  (1, 1) & \n \overset{\hbox{\tiny 0}}{\longrightarrow} \n & (2, 2)        & (0, 1) &&
  (2, 1) & \n \overset{\hbox{\tiny 1}}{\longrightarrow} \n & (0, 0)        & (0, 1) \\
  (2, 1) & \n \overset{\hbox{\tiny 0}}{\longrightarrow} \n & (0, 0)        & (0, 1) &&
  (2, 1) & \n \overset{\hbox{\tiny 1}}{\longrightarrow} \n & (2, 0)        & (2, 1) \\
  (2, 1) & \n \overset{\hbox{\tiny 0}}{\longrightarrow} \n & (2, 0)        & (2, 1) &&
  (2, 1) & \n \overset{\hbox{\tiny 1}}{\longrightarrow} \n & (2, 2)        & (2, 1) \\
  (2, 1) & \n \overset{\hbox{\tiny 0}}{\longrightarrow} \n & (2, 2)        & (2, 1) \end{array} $$
 Because~$\card (S) = 6$, there are~36 possible pairs of states at the tail and head of the arrows before an interaction but we have only listed the pairs for which an update indeed happens, i.e., the states at the head of the arrow before and after the interaction are different.
 Note that all the states generated by the coupling remain in~$S$ therefore~\eqref{eq:ordering} holds.
 In addition, Durrett and Swindle~\cite{Durrett_Swindle_1991} proved the following coexistence result for the grass-bush-tree system:
\begin{equation}
\label{eq:GBT-DS}
\begin{array}{rcl} \beta_2 > \beta_1^2 > 1 \quad \hbox{and} \quad M \ \hbox{large} & \Longrightarrow & \hbox{coexistence}. \end{array}
\end{equation}
 In particular, Theorem~\ref{th:GBT} follows from~\eqref{eq:GBT-DS} and the second implication in~\eqref{eq:ordering} which states that the~2s in the allelopathic model dominate the~2s in the grass-bush-tree system.

%%%%%%%%%%%%%%%%%%%%%%%%%%%%%%%%%%%%%%%%%%%%%%%%%%%%%%%%%%%%%%%%%%%%%%%%%%%%%%%%%%%%%%%%%%%%%%%%%%%%%%%%%%%%%%%%%%%%%%%%%%%%%%%%%%%%%%%%%%%%%%%%%%%%%%%%%%%%%%%%%%%%%%%%%%%%%%%%%%%%%%%%%%%%%%%%%%

\section{Proof of Theorem~\ref{th:MCP-dual}}
\label{sec:MCP-dual}
 This section is devoted to the proof of Theorem~\ref{th:MCP-dual} which states that, at least in~$d \geq 3$, when both species are equally fit, even weak inhibitory effects drive the susceptible species to extinction.
 The proof is mostly based on duality techniques even though, strictly speaking, the allelopathic model does not have a tractable dual process.
 The first step is to describe the dual processes of the contact process and the symmetric multitype contact process. \\
\indent
 To begin with, we think of the basic contact process with birth parameter~$\beta$ as being generated by the graphical representation that has unlabeled birth arrows from each site to each of their neighbors at rate~$\beta / N$ and crosses/death marks at each site at rate one.
 Given a realization of this graphical representation and two times~$0 < s < t$, we say that there is
 $$ \hbox{a path~$(y, t - s) \uparrow (x, t)$} \qquad \hbox{or} \qquad \hbox{a dual path~$(x, t) \downarrow (y, t - s)$} $$
 if we can go from one space-time point to the other one by moving forward/backward in time in the graphical representation, crossing the birth arrows in their direction/opposite of their direction, and avoiding the death marks.
 More precisely, there are sites and times
 $$ y = x_1, x_2, \ldots, x_n = x \quad \hbox{and} \quad t - s = s_0 < s_1 < \cdots < s_n = t $$
 such that the following two conditions hold:
\begin{itemize}
 \item there is a birth arrow~$(x_i, s_i) \to (x_{i + 1}, s_i)$ for all~$i = 1, 2, \ldots, n - 1$ and \vspace*{4pt}
 \item the segments~$\{x_i \} \times (s_{i - 1}, s_i)$ are void of death marks~$\times$ for~$i = 1, 2, \ldots, n$.
\end{itemize}
 The dual process starting at~$(x, t)$ is the set-valued process
\begin{equation}
\label{eq:dual}
\hat \xi_s^{(x, t)} = \{y \in \Z^d : \hbox{there is a dual path} \ (x, t) \downarrow (y, t - s) \} \quad \hbox{for all} \quad 0 \leq s \leq t.
\end{equation}
 By construction of the contact process from the graphical representation, a space-time point~$(x, t)$ is occupied if and only if it can be reached from a path starting at~$(y, 0)$ for some~$y$ that is initially occupied.
 This happens if and only if at least one of the dual paths starting at~$(x, t)$ lands on an occupied site, which gives the duality relationship
\begin{equation}
\label{eq:relationship}
\begin{array}{rcl} x \in \xi_t & \Longleftrightarrow & \hat \xi_t^{(x, t)} \cap \xi_0 \neq \varnothing \end{array}
\end{equation}
 where we have identified the configuration of the process with the set of occupied sites.
 In other words, the dual process keeps track of the potential ancestors of~$(x, t)$, and we refer to the left-hand side of Figure~\ref{fig:dual} for a picture.
 Note that the dual process exhibits a tree structure with branchings along the birth arrows.
 In our example,~$(x, t)$ is occupied if and only if at least one of the four sites in black at the bottom of the picture is occupied at time zero. \\
\indent
 We now give a description of the dual process of Neuhauser's multitype contact process.
 In the symmetric case~$\beta_1 = \beta_2 = \beta$, the process can be constructed from the same graphical representation as the basic contact process assuming that both types of individuals can give birth through the arrows.
 In particular, the duality relationship~\eqref{eq:relationship} still holds: space-time point~$(x, t)$ is occupied if and only if the dual process starting at this point intersects at least one occupied site in the initial configuration.
\begin{figure}[t!]
\centering
\scalebox{0.35}{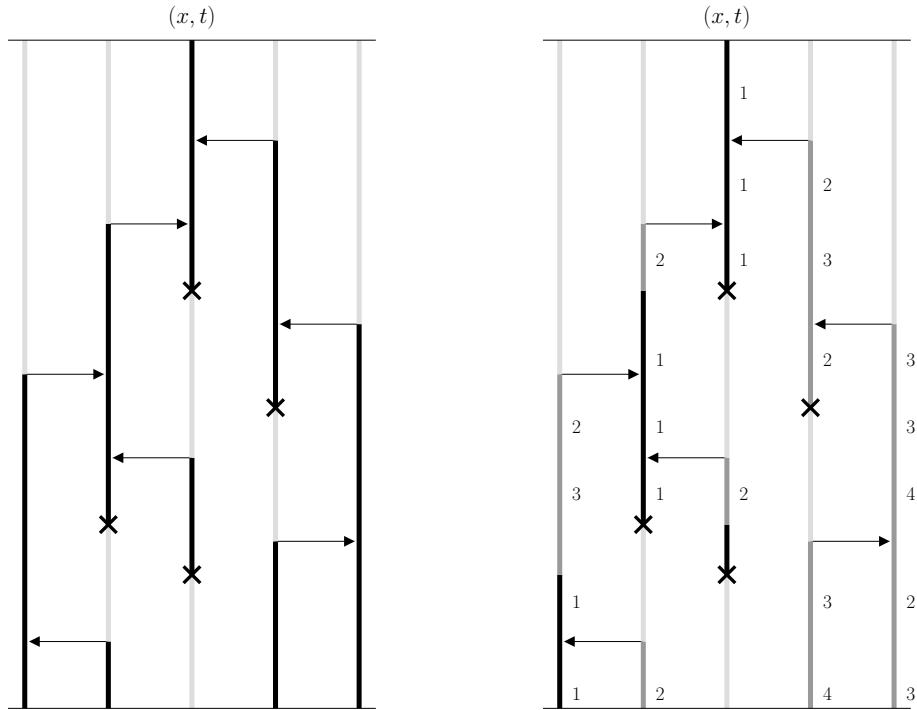}
\caption{\upshape{
 Dual process of the contact process on the left and dual process of the neutral multitype contact process along with the ancestor hierarchy on the right.
 In the right picture, the black line keeps track of the distinguished particle while the numbers show the ancestor hierarchy at various times.}}
\label{fig:dual}
\end{figure}
 However, when the intersection contains particles of different types, the type at point~$(x, t)$, as opposed to its occupancy, is unclear.
 Neuhauser~\cite{Neuhauser_1992} proved that the dual process of the symmetric multitype contact process consists of the dual process of the basic contact process along with an ancestor hierarchy in which the members can be arranged according to the order they determine the type of~$(x, t)$.
 We refer to Neuhauser~\cite{Neuhauser_1992} for a description of this ancestor hierarchy, and to Lanchier and Neuhauser~\cite{Lanchier_Neuhauser_2006} for a more formal definition using a labeling of the tree structure and the lexicographic order.
 The right-hand side of Figure~\ref{fig:dual} shows the ancestor hierarchy at various times.
 If the site labeled~1 at the bottom of the picture is initially occupied then~$(x, t)$ is of the same type as this site, if the site labeled~1 is initially empty but the site labeled~2 is initially occupied then~$(x, t)$ is of the same type as this site, and so on.
 In general,~$(x, t)$ is of the same type as the ancestor with the smallest index that is initially occupied. \\
\indent
 The first ancestor in the hierarchy plays a key role in the analysis of the multitype contact process and the allelopathic model.
 This ancestor is also called the distinguished particle, and its position at time~$t - s$ = dual time~$s$ is denoted by~$\hat \xi_s^{(x, t)} (1)$.
 Roughly speaking, the distinguished particle starts at site~$x$ and jumps to the lowest existing branch of the tree structure each time it meets a death mark going backward in time~(see the black path in Figure~\ref{fig:dual}).
 To study the evolution of the distinguished particle, it is convenient to assume that the exponential clocks in the graphical representation are defined for negative times so the process~\eqref{eq:dual} is well-defined for all~$s > 0$.
 Then, we say that~$(x, t)$ lives forever if the dual process starting at this point is nonempty for all times.
 Each time the distinguished particle jumps onto a point that lives forever, this point is called a renewal point.
 Because the distinguished particle is trapped in the~(infinite) subtree starting from a renewal point, its evolution before and after a renewal point are determined by disjoint parts of the graphical representation.
 In particular, the space-time displacements between consecutive renewal points are independent and identically distributed.
 Neuhauser also proved that these space-time displacements have exponentially decaying tails from which she deduced that, like simple symmetric random walks, the process that keeps track of the distinguished particle is recurrent in one and two dimensions and transient in higher dimensions. \\
\indent
 To study the allelopathic model when the two species are equally fit, we construct the process using the graphical representation of the neutral multitype contact process that has birth arrows and death marks, and also add like in Figure~\ref{fig:GR} type~0 arrows, that we call kill arrows from now on, from each site to each of their neighbors at rate~$\gamma / N$.
 The key to proving the theorem is to show that, when~$(x, t)$ lives forever, there is a sequence of distinguished particles that start at the tail of kill arrows pointing at the first ancestor of~$(x, t)$, live forever and never coalesce.
 To begin with, we follow the path of the distinguished particle starting at~$(x, t)$ going backward in time, let~$\sigma_i$ be the~$i$th time we cross the head of a birth arrow, let~$x_i$ be the location of the tail of this arrow, and let~$y_i$ be the location of the head of this arrow.
 Then, we define the events
 $$ \begin{array}{rcl}
     A_i & \n = \n & \hbox{there is one kill arrow~$x_i \to y_i$ in the time interval~$(\sigma_i - 1, \sigma_i)$ and no} \\
         & \n   \n & \hbox{crosses at or other arrows pointing at one of those two sites in this time interval}. \end{array} $$
 In case~$A_i$ occurs, we let~$\tau_i \in (\sigma_i - 1, \sigma_i)$ be the time at which the kill arrow appears.
\begin{lemma}
\label{lem:Ai}
 Assume that~$\gamma > 0$. Then, $P (\limsup_{i \to \infty} A_i) = 1$.
\end{lemma}
\begin{proof}
 Looking at the rates of the exponential clocks in the graphical representation and using that these exponential clocks are independent, we get
\begin{equation}
\label{eq:Ai1}
\begin{array}{rcl}
  P (A_i) & \n = \n & P (\poisson (\gamma / N) = 1) \times P (\poisson (2 (\gamma + \beta + 1) - \gamma / N) = 0) \vspace*{4pt} \\
          & \n = \n & (\gamma / N) \,\exp (- 2 (\gamma + \beta + 1)) > 0. \end{array}
\end{equation}
 In addition, whenever~$|\sigma_j - \sigma_i| > 1$, the two events~$A_i$ and~$A_j$ are independent, thus showing the existence of a subsequence of events that are independent and have a positive probability~\eqref{eq:Ai1}.
\begin{figure}[t!]
\centering
\scalebox{0.35}{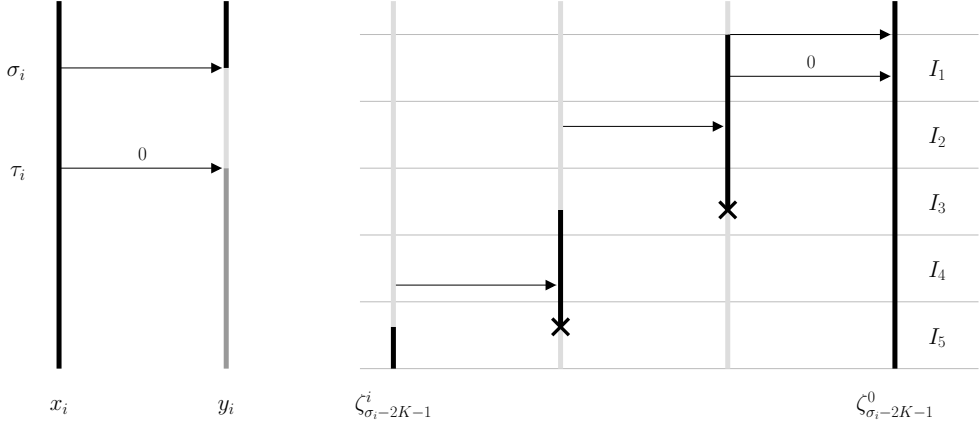}
\caption{\upshape{
 Picture of the event~$A_i$ on the left and of the event~$A_i \cap C_i$ on the right.}}
\label{fig:ACi}
\end{figure}
 In particular, the result follows from the second Borel-Cantelli lemma.
\end{proof} \\ \\
 In view of Lemma~\ref{lem:Ai}, there exists a subsequence~$A_{i_k}$ of events that all occur.
 We again call this subsequence~$A_i$ to avoid cumbersome notation.
 Note that, whenever~$(x_i, \sigma_i)$ is of type~1, this~1 must have killed earlier any potential~2 at~$(y_i, \tau_i)$ thanks to the presence of the kill arrow therefore~$(y_i, \sigma_i)$ is of type~1 as well.
 See the left-hand side of Figure~\ref{fig:ACi} for an illustration where as previously black means type~1 and gray means type~2.
 Because point~$(y_i, \sigma_i)$ is on the path of the distinguished particle starting at~$(x, t)$, this implies that~$(x, t)$ also is of type~1.
 In particular, to prove the theorem, it suffices to show that, whenever~$d \geq 3$,
\begin{equation}
\label{eq:Ai2}
\hbox{$(x_i, \sigma_i)$ is of type~1} \quad \hbox{for some} \quad i \in \N.
\end{equation}
 Now, define the new collection of events
 $$ \begin{array}{rcl} B_i & \n = \n & \hbox{the space-time point~$(x_i, \sigma_i)$ lives forever}. \end{array} $$
\begin{lemma}
\label{lem:Bi}
 Assume that~$\beta > \beta_c$. Then, $P (\limsup_{i \to \infty} B_i) = 1$.
\end{lemma}
\begin{proof}
 To begin with, note that, by self-duality of the contact process, the probability of~$B_i$ is equal to the probability of survival of the contact process with parameter~$\beta$, which is positive because~$\beta > \beta_c$.
 Now, assume that~$B_i$ does not occur.
 Letting~$j$ be the smallest integer larger than~$i$ such that the dual process starting at~$(x_i, \sigma_i)$ dies out in less than~$\sigma_i - \sigma_j$ units of time,~$B_j$ is determined by parts of the graphical representation that are disjoint from the parts of the graphical representation that determine the dual process starting at~$(x_i, \sigma_i)$.
 Because disjoint parts of the graphical representation are independent, we deduce that
 $$ P (B_j \,| \,B_i^c) = P (B_j) = P (\hat \xi_s^{(x_j, \sigma_j)} \neq \varnothing \ \hbox{for all} \ s) > 0. $$
 Using also that the events~$B_i$ are positively correlated implies that, with probability one, infinitely many of the events~$B_i$ occur, which proves the lemma.
\end{proof} \\ \\
 The lemma implies that there exists a subsequence~$A_{i_k} \cap B_{i_k}$ of events that all occur.
 As previously, we call this subsequence~$A_i \cap B_i$ to simplify the notation.
 Because the~$(x_i, \sigma_i)$ live forever, the distinguished particles starting from those points also live forever, and the next step is to study how these particles interact.
 Because these particles start at different times,
 $$ \xi_s^{(x_i, \sigma_i)} (1) \quad \hbox{and} \quad \xi_s^{(x_j, \sigma_j)} (1) \quad \hbox{for} \quad i \neq j $$
 do not refer to positions of different distinguished particles at the same time so, to facilitate the comparison, we reindex the processes using the actual time instead of the dual time.
 To also include the distinguished particle starting at~$(x, t)$ in the list, we let
 $$ \zeta_s^i = \hat \xi_{\sigma_i - s}^{(x_i, \sigma_i)} (1) \quad \hbox{for all} \quad i \geq 0 \quad \hbox{where} \quad (x_0, \sigma_0) = (x, t). $$
 The process~$\zeta_s^i$ is well-defined for all~$i \leq \sigma_i$, including negative~$s$ since the graphical representation has been extended to negative times.
 Now, observe that if the~$i$th distinguished particle lands on a~1 at time zero, meaning that~$\xi_0 (\zeta_0^i) = 1$, then space-time point~$(x_i, \sigma_i)$ must be of type~1.
 In particular, to prove~\eqref{eq:Ai2}, it suffices to show that
\begin{equation}
\label{eq:Bi1}
\xi_0 (\zeta_0^i) = 1 \quad \hbox{for some} \quad i \in \N.
\end{equation}
 To prove this, the idea is to use transience in dimensions~$d \geq 3$ to deduce that an arbitrarily large number of distinguished particles do not coalesce.
 More precisely, we will use the next result which follows from the proof of Neuhauser~\cite[Lemma~5.5]{Neuhauser_1992}.
\begin{lemma}
\label{lem:Neuhauser}
 Let~$d \geq 3$.
 There exists~$C > 0$ such that, for all~$K$ large,
 $$ \begin{array}{rcl} \norm{y - z} \geq K & \Longrightarrow & P (\norm{\hat \xi_s^{(y, t)} (1) - \hat \xi_s^{(z, t)} (1)} < s^{1/8} \ \hbox{for some} \ s) \leq CK^{- 1/10} + 2CK^{- 3/32}. \end{array} $$
\end{lemma}
 Motivated by this lemma, we define the events
 $$ \begin{array}{rcl} C_i & \n = \n & \hbox{the norm~$\norm{\zeta_{\sigma_i - 2K - 1}^i - \zeta_{\sigma_i - 2K - 1}^0}$ is larger than~$K$ but less than~$2K$}. \end{array} $$
\begin{lemma}
\label{lem:Ci}
 For all~$K > 0$, we have~$P (\limsup_{i \to \infty} C_i) = 1$.
\end{lemma}
\begin{proof}
 Subdividing~$(\sigma_i - 2K - 1, \sigma_i)$ into~$2K + 1$ intervals~$I_1, I_2, \ldots, I_{2K + 1}$ of length one going backward in time, the probability of~$C_i$ is bounded from below by the probability that
\begin{itemize}
 \item the distinguished particle~$\zeta_s^0$ does not jump between time~$\sigma_i$ and time~$\sigma_i - 2K - 1$, \vspace*{4pt}
 \item there are no crosses at and a single birth arrow pointing at~$\zeta_s^i$ whose tail is further away from the distinguished particle in the intervals~$I_2, I_4, \ldots, I_{2K}$, \vspace*{4pt}
 \item there are no birth arrows pointing at and one cross at~$\zeta_s^i$ in the intervals~$I_3, I_5, \ldots, I_{2K + 1}$.
\end{itemize}
 See the right-hand side of Figure~\ref{fig:ACi} for a picture of this more particular event.
 A direct calculation using the rate of occurrence of the birth arrows and crosses gives
\begin{equation}
\label{eq:Ci}
\begin{array}{rcl}
  P (C_i) & \n \geq \n & P (\poisson ((2K + 1)(\beta + 1)) = 0) \vspace*{4pt} \\
          & \n   \n & \hspace*{10pt} \times \,(P (\poisson (\beta / N) = 1) \times P (\poisson (\beta + 1 - \beta / N) = 0))^K \vspace*{4pt} \\
          & \n   \n & \hspace*{10pt} \times \,(P (\poisson (1) = 1) \times P (\poisson (\beta) = 0))^K \vspace*{4pt} \\
          & \n = \n & (\beta / N)^K \,\exp (- (4K + 1)(\beta + 1)) > 0. \end{array}
\end{equation}
 Note also that, when~$|\sigma_j - \sigma_i| > 2K + 1$, the two events~$A_i$ and~$A_j$ are determined by disjoint parts of the graphical representation and so are independent.
 In particular, there exists a subsequence of events that are independent and have a positive probability~\eqref{eq:Ci} therefore, like in the proof of~Lemma~\ref{lem:Ai}, we can conclude using the second Borel-Cantelli lemma.
\end{proof} \\ \\
 Lemmas~\ref{lem:Ai}, \ref{lem:Bi} and~\ref{lem:Ci} imply the existence of a sequence of space-time points, that we again denote by~$(x_i, \sigma_i)$ to simplify the notation, so that the events~$A_i \cap B_i \cap C_i$ occur.
 This, together with~Lemma~\ref{lem:Neuhauser}, implies that the number of distinguished particles~$\zeta_s^i$ that do not coalesce can be made arbitrarily large.
 More precisely, we have the following lemma.
\begin{lemma}
\label{lem:transience}
 Let~$d \geq 3$.
 Then, for all integers~$n > 0$,
 $$ \begin{array}{l} \lim_{s \to - \infty} \inf_{0 \leq k < l \leq n} \norm{\zeta_s^{i_k} - \zeta_s^{i_l}} = \infty \quad \hbox{for some} \quad 0 = i_0 < i_1 < i_2 < \ldots < i_n. \end{array} $$
\end{lemma}
\begin{proof}
 We prove the result by induction.
 Assuming that the statement of the lemma is true, for all integers~$K > 0$, there exists a time~$s_K > - \infty$ such that
 $$ \begin{array}{l} \norm{\zeta_s^{i_k} - \zeta_s^{i_l}} \geq 4K \quad \hbox{for all} \quad 0 \leq k < l \leq n \quad \hbox{and} \quad s < s_K. \end{array} $$
 In particular, for all~$j$ large enough that~$\sigma_j < s_K$, because the event~$C_j$ occurs,
 $$ \norm{\zeta_{\sigma_j - 2K - 1}^j - \zeta_{\sigma_j - 2K - 1}^{i_k}} \geq K \quad \hbox{for all} \quad 0 \leq k \leq n $$
 therefore it follows from Lemma~\ref{lem:Neuhauser} that the probability that the distinguished particle~$\zeta_s^j$ does not coalesce with any of the other~$n + 1$ distinguished particles is larger than
 $$ \begin{array}{l}
      P (\lim_{s \to - \infty} \norm{\zeta_s^j - \zeta_s^{i_k}} = \infty \ \hbox{for all} \ 0 \leq k \leq n) \vspace*{4pt} \\ \hspace*{25pt} =
      1 - P (\lim_{s \to - \infty} \norm{\zeta_s^j - \zeta_s^{i_k}} \neq \infty \ \hbox{for some} \ 0 \leq k \leq n) \vspace*{4pt} \\ \hspace*{25pt} \geq
      1 - (n + 1) CK^{- 1/10} - 2 (n + 1) CK^{- 3/32} \end{array} $$
 which can be made positive by choosing~$K$ sufficiently large.
 In particular, after a geometric number of trials, we find a particle~$\zeta_s^j$ whose distance from the other~$n + 1$ particles goes to infinity.
 This shows that the statement of the lemma still holds replacing~$n$ with~$n + 1$.
\end{proof} \\ \\
 To complete the proof of the theorem, we let
 $$ \Theta_t = \{\zeta_0^i : i \geq 0 \ \hbox{and} \ \sigma_i > 0 \} $$
 be the set of sites occupied by the distinguished particles of the~$(x_i, \sigma_i)$ that are contained in the positive time half-space.
 According to Lemma~\ref{lem:transience}, this set can be made arbitrarily large by choosing~$t$ large enough.
 In particular, starting from a translation invariant distribution with a positive density of particles of type~1, it follows from Harris~\cite[Lemma~9.14]{Harris_1976} that
 $$ \begin{array}{l} \lim_{t \to \infty} P (\Theta_t \cap \{x : \xi_0 (x) = 1 \} \neq \varnothing) = 1. \end{array} $$
 This shows that~\eqref{eq:Bi1} occurs with probability one in the~$t \to \infty$ limit which, in turn, implies that~\eqref{eq:Ai2} occurs with probability one in the limit, and proves the theorem.

%%%%%%%%%%%%%%%%%%%%%%%%%%%%%%%%%%%%%%%%%%%%%%%%%%%%%%%%%%%%%%%%%%%%%%%%%%%%%%%%%%%%%%%%%%%%%%%%%%%%%%%%%%%%%%%%%%%%%%%%%%%%%%%%%%%%%%%%%%%%%%%%%%%%%%%%%%%%%%%%%%%%%%%%%%%%%%%%%%%%%%%%%%%%%%%%%%

\bibliographystyle{plain}
\bibliography{biblio.bib}

\end{document}